\documentclass{article}

\usepackage{vmargin}
\usepackage{algpseudocode}
\usepackage{amsmath}
\usepackage{amsthm}
\usepackage[utf8]{inputenc}
\usepackage{graphicx}
\usepackage{enumerate,dsfont}
\usepackage{caption}
\usepackage{marvosym}
\usepackage{xcolor}
\usepackage{amssymb}
\usepackage{amsbsy,url}
\usepackage[normalem]{ulem}

\newtheorem{theorem}{Theorem}[section]
\newtheorem{proposition}[theorem]{Proposition}
\newtheorem{corollary}[theorem]{Corollary}
\newtheorem{lemma}[theorem]{Lemma}

\newcounter{hypA}
\newenvironment{hypA}{\refstepcounter{hypA}\begin{itemize}
  \item[{\bf A\arabic{hypA}}]}{\end{itemize}}

\newcounter{hypE}

\newcommand{\rmd}{\mathrm{d}}

\def\calC{\mathcal{C}}

\def\calB{\mathcal{B}}

\def\Obs{Y}
\def\Prop{Z}

\newcommand{\eqdef}{\ensuremath{\stackrel{\mathrm{def}}{=}}}
\def\eqsp{\;}

\newcommand{\PE}{\mathbb{E}}
\newcommand{\PP}{\mathbb{P}}

\newcommand{\asinh}{\operatorname{asinh}}

\def\Rset{\mathbb{R}}
\def\Nset{\mathbb{N}}

\def\un{\mathds{1}}
\def\1{\mathds{1}}
\newcommand\indi[1]{\1_{\{#1\}}}

\newcommand{\dimstate}{p}

\newcommand{\pscal}[2]{\left\langle #1, #2 \right\rangle}

\newcommand{\M}{\mathcal{M}}

\newcommand{\pirp}{\pi}

\newcommand{\cnoise}{\tau}

\newcommand{\ccone}{\epsilon}

\newcommand{\cgaus}{k_2}
\newcommand{\vgaus}{\sigma_2}
\newcommand{\cigaus}{k_1}
\newcommand{\vigaus}{\sigma_1}
\newcommand{\dom}{\rmd \nu}

\newcommand{\sign}{\ \textrm{sign}}

\newcommand{\expvec}[2] {#1^{[#2]}}
\newcommand{\shrinkvec}[2] {#1_{[#2]}}
\newcommand{\idxmodel}[1]{\kappa^{[#1]}}

\newcommand{\design}{G}
\newcommand{\state}{X}
\newcommand{\noise}{E}
\newcommand{\noiseprop}{\xi}
\newcommand{\constV}{\beta}
\newcommand{\trunckernel}{P}
\newcommand{\sto}{\Psi}
\newcommand{\stog}{\Psi_1}

\newcommand{\stos}{\Psi_2}

\def\gauss{\mathcal{N}}

\newcommand{\coint}[1]{\left[#1\right)}

\newcommand{\ooint}[1]{\left(#1\right)}

\def\iff{if and only if} 
\def\hstos{t} 

\begin{document}

\title{A Shrinkage-Thresholding Metropolis adjusted Langevin algorithm for Bayesian Variable Selection}

\author{Amandine Schreck\footnotemark[1], Gersende Fort\footnotemark[1], 
Sylvain Le Corff\footnotemark[2],
Eric Moulines\footnotemark[1]}

\footnotetext[1]{LTCI, Telecom ParisTech \& CNRS, 46 rue Barrault 75634 Paris Cedex 13 France}
\footnotetext[2]{Laboratoire de Math\'ematiques, Universit\'e Paris-Sud and CNRS, UMR 8628, Orsay, France}

\maketitle

\begin{abstract}
This paper introduces a new Markov Chain Monte Carlo method for Bayesian
  variable selection in high dimensional settings. The algorithm is a
  Hastings-Metropolis sampler with a proposal mechanism which combines a
  Metropolis Adjusted Langevin (MALA) step to propose local moves associated
  with a shrinkage-thresholding step allowing to propose new models.  The
  geometric ergodicity of this new trans-dimensional Markov Chain Monte Carlo
  sampler is established.  An extensive numerical experiment, on simulated and
  real data, is presented to illustrate the performance of the proposed
  algorithm in comparison with some more classical trans-dimensional
  algorithms.
\end{abstract}

\section{Introduction}
We focus on variable selection
in regression problems: the objective is to explain a
response variable with a (possibly very) large number of explanatory variables,
which can be either discrete or continuous. In many applications, it is known that only a small
fraction of explanatory variables explains a large fraction of the
observations, and using this information is crucial for inference. 
Variable selection is particularly challenging in high dimensional settings.

A variety of algorithms to explore the collection of models and criteria for
selecting among competing models has been proposed.  In the Bayesian framework,
the variable selection problem is transformed into posterior inference: rather
than searching a highly hypothetical "best" model, Bayesian analysis attempts
to estimate the joint posterior distribution of the collection of all subsets
of parameters. In high dimension, this aim is often overly ambitious:
estimating the marginal posterior probability that a variable should be
included in the model is already challenging.

In the last three decades, Markov Chain Monte Carlo (MCMC) methods have been the most
commonly used computational procedures to sample posterior distributions
\cite{robert:casella:1999}. 
An early attempt to perform variable selection is the Reversible Jump MCMC
(RJMCMC) introduced in \cite{green:1995}.
RJMCMC is a trans-dimensional sampler which produces a
Markov chain evolving between spaces of different dimensions. The dimension of
the sample varies at each iteration as active (nonzero) parameters are added or discarded
from the model. Each new sample is accepted or rejected using a
Metropolis-Hastings step where the acceptance probability is adjusted to the
trans-dimensional moves. RJMCMC requires ingenuity in designing appropriate
jumping rules to produce computationally efficient and theoretically effective
methods. Despite many
attempts~\cite{brooks:etal:2003,karagiannis:andrieu:2013}, this algorithm is
prone to fail when the dimension of the parameter space is large (as illustrated in our numerical section).

\cite{carlin:chib:1995} considers another setting that encompasses all the
models jointly: at each iteration, pseudo-prior distributions are used to
jointly sample regression parameters associated with all models. For high
dimensional statistical problems, sampling jointly all models is of course out
of reach. A more efficient algorithm, the Metropolized Carlin and Chib (MCC),
simultaneously proposed by
\cite{godsill:2001,dellaportas:forster:ntzoufras:2002} and later improved
by~\cite{petralias:dellaportas:2013}, does not require to sample from the whole
collection of models and therefore can be implemented in practice. The mixing
of this algorithm depends critically on the specification of pseudo-priors,
which requires also a fair amount of tuning.

Other MCMC approaches for Bayesian variable selection define a posterior distribution on the model space, where a model is a binary vector locating the active (nonzero) components of the regression vector. The objective is 
to estimate probabilities of activation for each regression parameter. In~\cite{brown:fearn:vannucci:2001} for example, this exploration is performed with a Gibbs sampler. Variants and adaptive versions of the Gibbs sampler for this problem have been proposed in \cite{nott:kohn:2005,lamnisos:etal:2013}. Samples from the posterior distribution of the models are obtained in~\cite{shi:dunson:2011} and in \cite{schafer:chopin:2013} with particle filters.

In this paper, we introduce a novel algorithm, the Shrinkage-Thresholding
Metropolis-Adjusted Langevin Algorithm (STMALA) to perform sparse regression in
high dimensional models.  This algorithm might be seen as a trans-dimensional
MCMC method relying on the MALA algorithm (see \cite{roberts:tweedie:1996b}).
The proposal distribution in the STMALA algorithm goes as follows:
\begin{itemize}
\item compute  a noisy gradient step  of the logarithm of  the smooth part of the target distribution;
\item apply a shrinkage-thresholding operator to ensure sparsity and to shrink
  values of the regression parameters toward zero;
\item use an accept-reject step  to guarantee the convergence to the correct target distribution.
\end{itemize}
Each iteration of the STMALA algorithm may be seen as a randomized version of
the Shrinkage-Thresholding algorithm (see~\cite{beck:teboulle:2009}) to guide
variable selection.  The Shrinkage-Thresholding algorithm (and its accelerated
version FISTA) is one of the most effective method to solve sparse inverse
problems.  Our intuition is that a single iteration of the
Shrinkage-Thresholding algorithm (with some additional noise added to ensure
irreducibility) is a sensible way to visit collection of models.  This
intuition is supported both by very promising experimental results obtained in
a variety of challenging situations and by some theoretical results. In
particular, we have established the geometric ergodicity of the STMALA
algorithm for a large class of target distributions. To our best knowledge, it
is the first result providing a rate of convergence for a trans-dimensional
MCMC algorithm (like RJMCMC and MCC); usually, only Harris recurrence is
proved, see \cite{roberts:rosenthal:2006}.

Our algorithm is closely related to the proximal MCMC algorithm of
\cite{pereyra:2015}; the main difference stems from the fact that our algorithm
is designed to sample jointly the models and their parameters, whereas
\cite{pereyra:2015} is a method to sample from high-dimensional posterior
distributions with sparsity inducing priors.

This paper is organized as follows. STMALA and its application to Bayesian
variable selection is described in Section~\ref{sec:PMALA}.  The geometric
ergodicity of the STMALA algorithm is addressed in
Section~\ref{sec:ergodicity}.  Numerical experiments on simulated and real data
sets to assess the performance of STMALA are given in Section~\ref{sec:exp}.
All the proofs are postponed to Section~\ref{PMALA:sec:proofs}.

\section{The Shrinkage-Thresholding MALA algorithm}
\label{sec:PMALA}
This section introduces the Shrinkage-Thresholding MALA algorithm which is
designed to sample from a target distribution defined on $\Rset^{\dimstate}$,
$\dimstate\in \Nset^*$. Denote by $\M
\eqdef\{0,1\}^\dimstate$ the set of binary vectors.  For any $m = (m_1, \dots,
m_\dimstate) \in \M$, set
\begin{equation}
\label{eq:def:Im}
I_m \eqdef \{i \in \{1, \cdots, \dimstate\};\;  m_i = 1 \} \eqsp,
\end{equation}
the family of active, i.e. nonzero, variables.  
For any $m \in \M$, denote by $S_m$ the subset of $\Rset^{\dimstate}$ defined
by
\begin{equation}
\label{eq:def:sm}
 S_m \eqdef \{z \in \Rset^{\dimstate}, z_{i} \neq 0, i \in I_m, z_{j} = 0, j \not \in I_m \}
\end{equation}
and by $|m| \eqdef \sum_{i=1}^\dimstate m_i$ the number of non-zero components
in $m$.  $\{S_m, m \in \M\}$ is a partition of $\Rset^\dimstate$ and we assume
that the target distribution may be written as
\begin{equation}
\label{eq:target-distribution}
\pi( \rmd x) = \sum_{m \in \M} \omega_m \pi_m(x) \un_{S_m}(x)   \nu_m(\rmd x) \eqsp,
\end{equation}
where $\{\omega_m, m \in \M\}$ is the prior probability of the models and
$\pi_m(x) \nu_m(\rmd x)$ is the distribution of $x$ conditionally to the model
$m$. We consider situations when $\nu_m(\rmd x) = \prod_{i \in I_m} \rmd x_i
\prod_{j \notin I_m} \delta_0(\rmd x_j)$ and $\pi_m(x) \propto \exp(- U_m(x) -
V_m(x))$ with $x \mapsto U_m(x)$ continuously differentiable and $x \mapsto V_m(x)$ possibly
non-smooth (a penalization term).

Two different shrinkage-thresholding operators are considered to sample sparse
vectors, namely the Proximal one (Prox) $\stog$ and the soft thresholding
operator with vanishing shrinkage (STVS) $\stos$: for any $\gamma>0$,
$1\leq i\leq \dimstate$ and $u = (u_1, \cdots, u_\dimstate) \in
\Rset^\dimstate$,
\begin{align}
\label{eq:expression-psi}
  &\left( \stog(u) \right)_{i} \eqdef  u_{i} \left(1- \gamma/|u_{i}|
  \right)_+\eqsp, \\
  &\left( \stos(u) \right)_{i} \eqdef u_{i} \left( 1 -  \gamma^2/|u_{i}|^2 \right)_+ \eqsp,
\end{align}
where for $a \in \Rset$, $a_+$ denotes the positive part of $a$: $a_+ \eqdef \max(a,0)$.
\begin{figure}[ht!]
 \begin{center}
   \includegraphics[scale=.3]{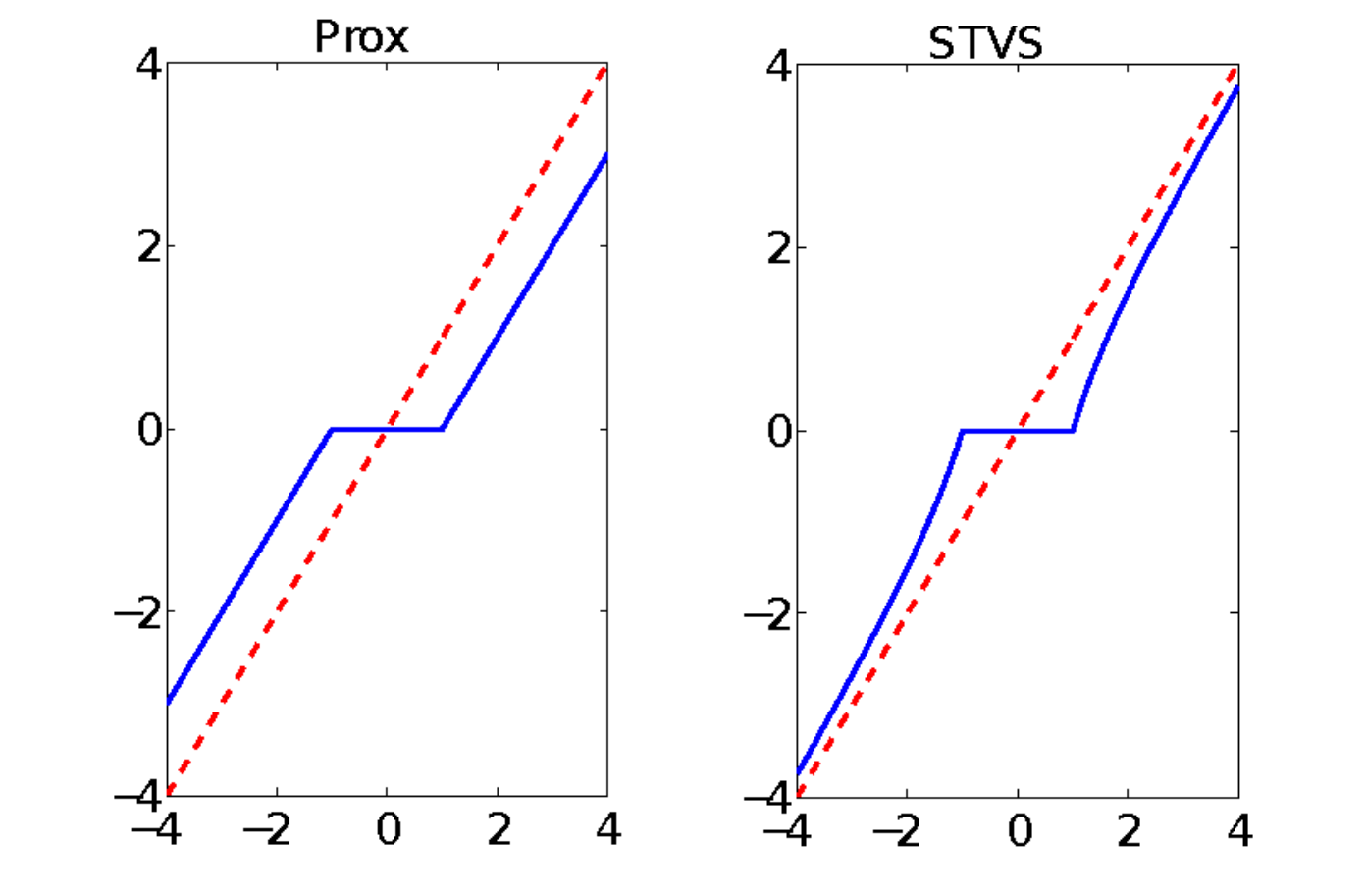}
 \end{center}
     \caption{Continuous line : Shrinkage-Thresholding functions associated with Prox (left) and STVS (right) in one dimension. Dashed lines : the identity function.} \label{fig:toyex:fct_seuil:d1}
\end{figure}
Lemma~\ref{STMALA:prox} shows that the soft thresholding operator with
vanishing shrinkage $\stos$, known as the empirical Wiener operator
(see~\cite{siedenburg:2012}), compromises between minimizing a (non-convex)
function $h$ and being close to $u$.
\begin{lemma} \label{STMALA:prox}
For any $\gamma>0$ and $u \in \Rset^\dimstate$,
\[
\stos(u) = \mathrm{argmin}_{x \in \Rset^\dimstate} \left( \hstos(x)  +  \frac{1}{2} \|x -u\|^2 \right)\eqsp,
\]
where
 \begin{equation*}
\hstos(x) = \gamma^2 \left[ \asinh \left( \|x\|/(2 \gamma) \right) - (1/2) \exp \left(-2   \asinh \left( \|x\|/(2 \gamma) \right) \right) \right] \eqsp.
 \end{equation*}
\end{lemma}

\begin{proof}
The proof  is postponed to Section~\ref{proof:lem:STMAL:prox}.
\end{proof}
From a current state $ \state^{n} \in \Rset^\dimstate$ the algorithm proposes a
new point $\Prop$ defined by
\begin{equation}
\label{eq:proposal:algo}
\Prop = \sto \left( \state^n + \frac{\sigma^2}{2} h(\state^n) + \sigma \noiseprop^{n+1}  \right)\eqsp,
\end{equation}
where $\sigma > 0$, $\noiseprop^{n+1}\sim \gauss_\dimstate(0,I)$ and for all $x
\in \Rset^\dimstate$ and $D > 0$,
\begin{equation}
\label{eq:definition-h}
h(x) \eqdef \sum_{m \in \M} \1_{S_m}(x) \frac{D \, \nabla U_m(x)}{D \vee \left\|  \nabla U_m(x) \right\|}\eqsp,
\end{equation}
with $a \vee b= \max(a,b)$.  The following lemma shows that this proposal
mecanism is equivalent to sampling a new binary vector $m'\in \M$ conditionally
to $x$; and then sampling a new vector with non null components in
$\Rset^{|m'|}$ conditionally to $(m',x)$.  Define
\begin{equation}
\label{eq:local:mu}
\mu(x) \eqdef x + \sigma^2 h(x)/2 \eqsp.
\end{equation}
\begin{lemma}
\label{PMALA:lem:propnorm21}
Let $x\in\Rset^\dimstate$, $D,\gamma,\sigma>0$. Let $\sto \in \{\sto_1, \sto_2
\}$. The random vector $ \sto(\mu(x) + \sigma \noiseprop )$ where $\noiseprop
\sim \gauss_\dimstate(0,I)$ has a density with respect to $ \sum_{m \in \M}
\nu_m$ given by $z \mapsto q_{\sto}(x,z)$ with
\begin{equation} \label{eq:proposal:L21gradientproximal}
  q_{\sto}(x,z) \1_{S_m}(z)= \left( \prod_{i \notin I_m} \rho( \mu_{i}(x) )
  \right) \times \left( \prod_{i \in I_m} f_{\sto}( \mu_{i}(x), z_{i}) \right)
\end{equation}
where for any $c\in \Rset$ and $z \in \Rset^\star$
\begin{align*}
&\rho(c) \eqdef \PP\left\{|c+\sigma \zeta| \leq \gamma\right\}\eqsp,\; \mathrm{with}\; \zeta\sim \gauss(0, 1)\eqsp,\\
&f_{\stog}(c, z) \eqdef\left( 2 \pi \sigma^2 \right)^{-1/2}  \times \exp \left\{-\left| \left( 1 + \gamma |z|^{-1}
      \right){z}-c\right|^2 / (2\sigma^2) \right\}\eqsp;
\end{align*}
and
\begin{equation*}
f_{\stos}(c, z) \eqdef\left( 2 \pi \sigma^2 \right)^{-1/2}
    g\left(\gamma^2 |z|^{-2}\right)  \tilde g\left(\gamma^2 |z|^{-2}\right)  \exp \left(-\left| g\left(\gamma^2 |z|^{-2}\right) \ {z}-c\right|^2/(2\sigma^2) \right) \  \eqsp,
\end{equation*}
with
\[
g(u)\eqdef \frac12 (1+\sqrt{1 + 4u}) \eqsp,  \tilde g(u) \eqdef
  \frac{1}{\sqrt{1+4u}} \eqsp,
\]
\end{lemma}
\begin{proof}
The proof  is postponed to Section~\ref{sec:proofs:pmala:etc}.
\end{proof}
For any $x$, $z \mapsto q_\sto(x,z)$ consists in
\begin{enumerate}[(i)]
\item sampling each component of a new model $m' \in \M$ as independent
  $\{0,1\}$-Bernoulli random variable with success parameter $\rho(\mu_i(x))$,
  $1 \leq i \leq p$;
\item for $i \notin I_{m'}$, set $z_{i} = 0$; conditionally to $(m', x)$,
  sample independent components such that for any $i \in I_{m'}$, the
  distribution of $z_{i}$ on $\Rset^{\star}$ is $f_{\sto}(\mu_i(x),z_i)$.
\end{enumerate}
The proposal \eqref{eq:proposal:algo} is then accepted and $X^{n+1} =Z$ with
probability $\alpha_\sto(\state^n, \Prop)$ given by
\begin{equation}\label{eq:expression:alpha}
\alpha_{\sto}(x, z) \eqdef 1 \wedge \frac{ \pirp(z) \, q_{\sto}(z, x)}{\pirp(x)
  \, q_{\sto}(x, z)} \eqsp;
\end{equation}
otherwise, $X^{n+1} = X^{n}$.  In high dimensional settings, STMALA may
encounter some difficulties to accept the proposed moves. Following
\cite{neal:roberts:2006}, we introduce a variant of the algorithm in which only
a fixed number $\eta$ of components of $\state^n$ is updated at each iteration
$n$. This is achieved by combining STMALA and a Gibbs sampler in a
STMALA-within-Gibbs algorithm.

\section{$V$-Geometric ergodicity of the $L_{1}$ proximal STMALA}
\label{sec:ergodicity}
In this section, we address the $V$-geometric ergodicity of the STMALA chain
$(X^n)_{n\geq 0}$ under the following assumptions: for any $m \in \M$,
\begin{hypA} \label{hyp:reg:pi}
 \begin{enumerate}[(i)]
 \item \label{hyp:reg:pi:pos} $\omega_m>0$ and $\pi_m>0$ on $S_m$.
 \item \label{hyp:reg:pi:cont} $\pi_m$ is continuous on $S_m$.
 \item \label{hyp:reg:pi:lim} $\pi_m(x) \1_{S_m}(x) \to 0$ when $\|x\| \to
   \infty$.
 \end{enumerate}
\end{hypA}
\begin{hypA} \label{hyp:superexp:pi}
  for any $s>0$,
\begin{align*}
  \lim_{r \to \infty} \sup \limits_{x \in S_m, \|x\| \geq r} \pi_m(x + s
    \, n(x))/\pi_m(x) = 0\eqsp,
\end{align*}
where $n(x) \eqdef x/\|x\|$.
\end{hypA}
Let $b,\epsilon>0$ and $u \in (0, b)$. For any $m \in \M$ and $x \in S_m$,
define
 \begin{equation} \label{eq:definition:cone}
   W_m(x) \eqdef \{ (\|x\|-u) n(x) - s \zeta: s \in (0,b-u)\;;\;\zeta \in S_m, \|\zeta\| =1, \|\zeta-n(x)\| \leq \ccone \} \eqsp.
 \end{equation}
 $W_m(x)$ is the cone of $S_m$ with apex $x - u \ n(x)$ and aperture $2
 \epsilon$.  We will prove (see Lemma~\ref{lemme:cone}) that A\ref{hyp:cone:pi}
 guarantees that,  the probability to accept a move
 from $x$ to any point of $W_m(x)$ converges to one as $\| x \|$ goes to infinity.
\begin{hypA} \label{hyp:cone:pi}
  There exist $b,R,\ccone>0$ and $u \in (0,b)$ such that for any $m \in \M$,
  for any $x \in S_m \cap \{\|x\| \geq R\}$, for all $y \in S_m \cap W_m(x)$:
  $\pi_m(x - u \ n(x)) \leq \pi_m(y)$.
\end{hypA}
When for any $m \in \M$, $\pi_m$ is differentiable on $S_m$,
A\ref{hyp:superexp:pi} and A\ref{hyp:cone:pi} are satisfied if
(see for details), for all $m \in \M$,
\begin{align*} 
  & \lim \limits_{x \in S_m, \|x\| \to \infty} \pscal{n(x)}{\nabla \log (\pi_m (x))} = - \infty \eqsp, \\
  & \limsup \limits_{x \in S_m, \|x\| \to \infty} \pscal{n(x)}{n(\nabla
    \pi_m(x))} <0 \eqsp;
\end{align*}
(see ~\cite[Section 4 and the proof of Theorem 4.3]{jarner:hansen:2000} for
details).
 
Let $\trunckernel_{\sto}$ denote the transition kernel associated to the
Hastings-Metropolis move with proposal (\ref{eq:proposal:algo}) and
acceptance-rejection ratio (\ref{eq:expression:alpha}).
\begin{theorem}
\label{th:erggeo}
Assume A\ref{hyp:reg:pi}-\ref{hyp:cone:pi} hold. Then, for any $\sto \in
\{\sto_1, \sto_2\}$, for any $\constV \in (0,1)$, there exist $C >0$ and
$\lambda \in (0,1)$ such that for any $ n \geq 0$ and any $x \in
\Rset^\dimstate$,
 \begin{align}
  \|\trunckernel_{\sto}^n(x,.) - \pi \|_V \leq C \, V(x) \, \lambda^n \eqsp,
 \end{align}
 where $V(x) \propto \pi(x)^{-\constV}$ and for any signed measure $\eta$,
 $\|\eta\|_V \eqdef \sup \limits_{f, |f| \leq V}|\eta(f)|$.
\end{theorem}
\begin{proof}
  By definition of the acceptance-rejection ratio, $\pi$ is invariant with
  respect to $\trunckernel_{\sto}$.  The $V$-uniform geometric ergodicity
  follows from Proposition~\ref{prop:small} and Proposition~\ref{prop:drift}
  given in Section~\ref{sec:proofs:erg}: Proposition~\ref{prop:small}
  establishes that the chain is psi-irreducible and aperiodic and shows that
  any Borel set $C \subset \Rset^\dimstate$ such that $C \cap S_m$ is a compact
  subset of $S_m$ is a small set for $\trunckernel_{\sto}$;
  Proposition~\ref{prop:drift} shows that there exists an accessible small set
  $C \subset \Rset^\dimstate$ and constants $c_1 \in (0,1)$ and $c_2<\infty$
  such that for any $x \in \Rset^\dimstate$, $\trunckernel_{\sto} V(x) \leq c_1
  V(x) + c_2 \un_{C}(x)$.  The proof is then concluded by \cite[Theorem
  15.0.2]{meyn:tweedie:1993}.
\end{proof}

\section{Numerical illustrations}
\label{sec:exp}
In this section, STMALA\footnote{MATLAB codes for STMALA are available at the
  address http://www.math.u-psud.fr/$\sim$lecorff/software.html} is compared to
the reversible jump Markov chain Monte Carlo (RJMCMC) algorithm. For any
$\ell\times\ell'$ matrix $A$ and any $1\le j \le \ell$, $1\le k \le \ell'$,
$A_{\cdot,k}$ (resp. $A_{j,\cdot}$) denotes the $k$-th column (resp. the $j$-th
row) of $A$. In all the sequel, only the performance of STMALA with $\stos$ is
considered due to lack of space. It has been experimentally observed in all the
considered scenarios that $\stos$ performs significantly better than $\stog$,
because it avoids to shrink the significative components of $x$.

In the examples below, $\pi$ is the posterior distribution of a regression
vector in a logistic regression model; $\pi_m$ is the conditional distribution
of the regression vector conditionally to the observations and to the model
$m$.

\subsection{Logistic regression}
Let $G$ be a known $N\times \dimstate$ design matrix.  We have $N$ independent observations $Y =
(Y_1,\ldots,Y_N)$ such that for all $i$, $Y_i$ is
a Bernoulli random variable with parameter
$\exp(G_{i,\cdot}X)/(1+\exp(G_{i,\cdot}X))$.  Conditionally to a model $m \in
\M$, the prior on the nonzero components of the regression vector $X \in S_m$
is $\gauss(0,c(G'_{m}G_{m})^{-1})$, where $c$ is a known scaling parameter, and
$G_{m}$ denotes the matrix with columns $\{G_{\cdot,i}, i \in I_m \}$.  The
prior on the models $\omega_m$ is equal to $\theta_\star^{|m|}
(1-\theta_\star)^{p-|m|}$ for $\theta_\star \in (0,1)$. In this experiment, we
choose $\dimstate = 50$ and $N = 100$ to assess the performance of STMALA in a
simple framework; the components of $G$ are i.i.d. $\gauss(0,1)$ and
$\theta_{\star} = 0.05$. The algorithm is run with $c = 100$, $\sigma = 0.3$
and $\eta = 5$. The choice of the threshold $\gamma$ in $\sto_2$ is crucial (if
$\gamma$ is too large, few nonzero samples are proposed and the algorithm
converges slowly and if $\gamma$ is too small, the algorithm proposes
non-sparse solutions that are not likely to be accepted): $\gamma$ is set to
$0.4$ to get a mean acceptance rate of around $20\%$.

STMALA is used to estimate the posterior probabilities of activation of the
components of $\state$, defined for all $1 \leq i \leq \dimstate$ as the
conditional probability of the event $\{X_i\ne 0\}$ given the observations $Y$.
The estimation is given by $\sum_{n=B}^{N_{it}+B} \indi{\state^n_i \neq
  0}/N_{it}$ where $N_{it}$ is the number of iterations of the algorithm and
$B$ denotes the number of iterations discarded as a burn-in period. We choose
$N_{it} = 50.000$ and $B = 10.000$.  Figure~\ref{fig:logistic_x_acc} (top)
provides the true regression vector, the posterior mean of the regression
vector given by STMALA and the estimated activation probabilities over $100$
independent Monte Carlo runs. This experiment highlights the ability of STMALA
to choose the good model (the $3$ nonzero components of $X$ are recovered) and
to get high posterior probabilities of activation for the selected components
of $X$.

\begin{figure}[ht!]
\begin{center}
  \includegraphics[width=.9\linewidth]{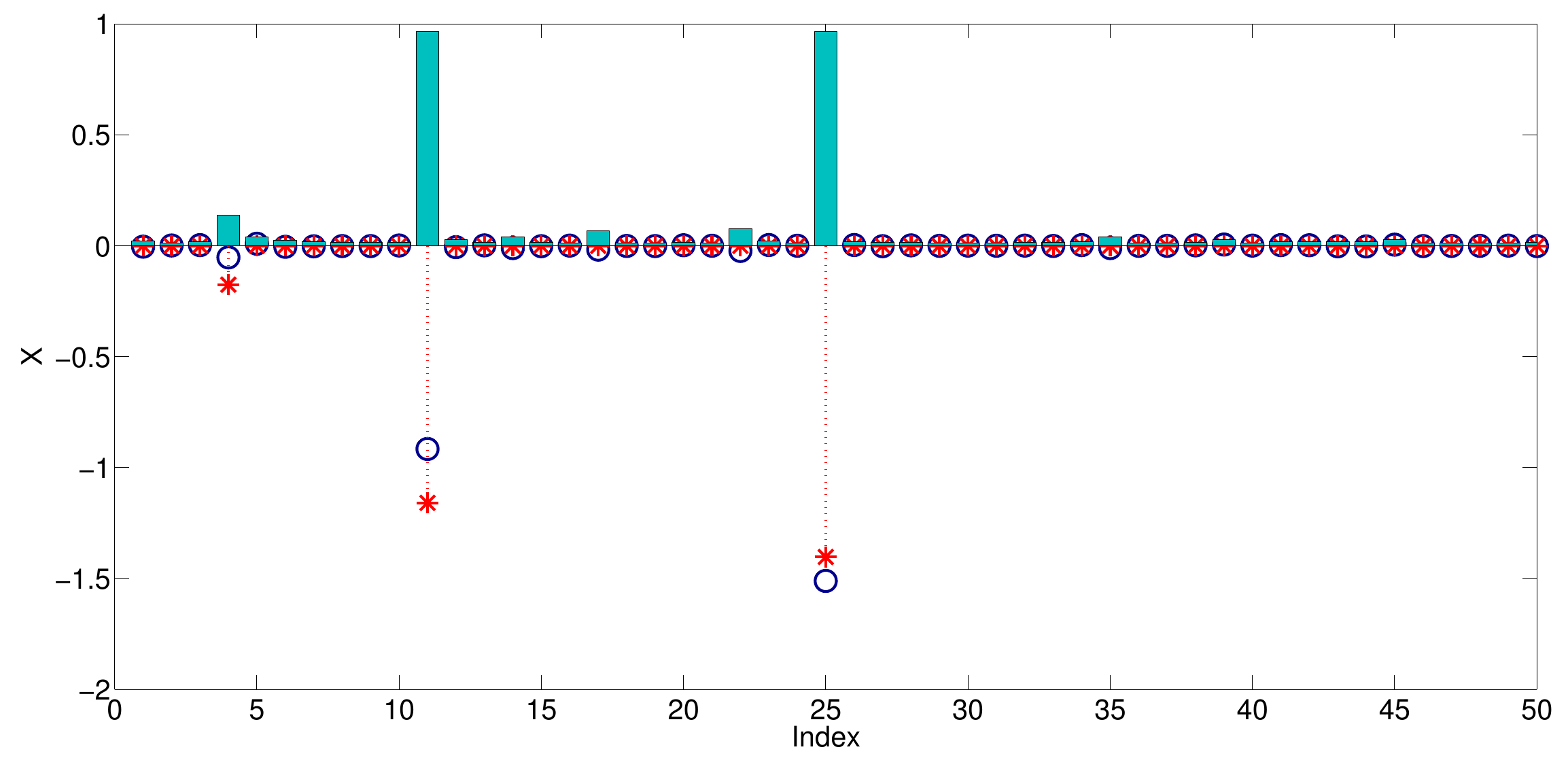}
  \includegraphics[width=.9\linewidth]{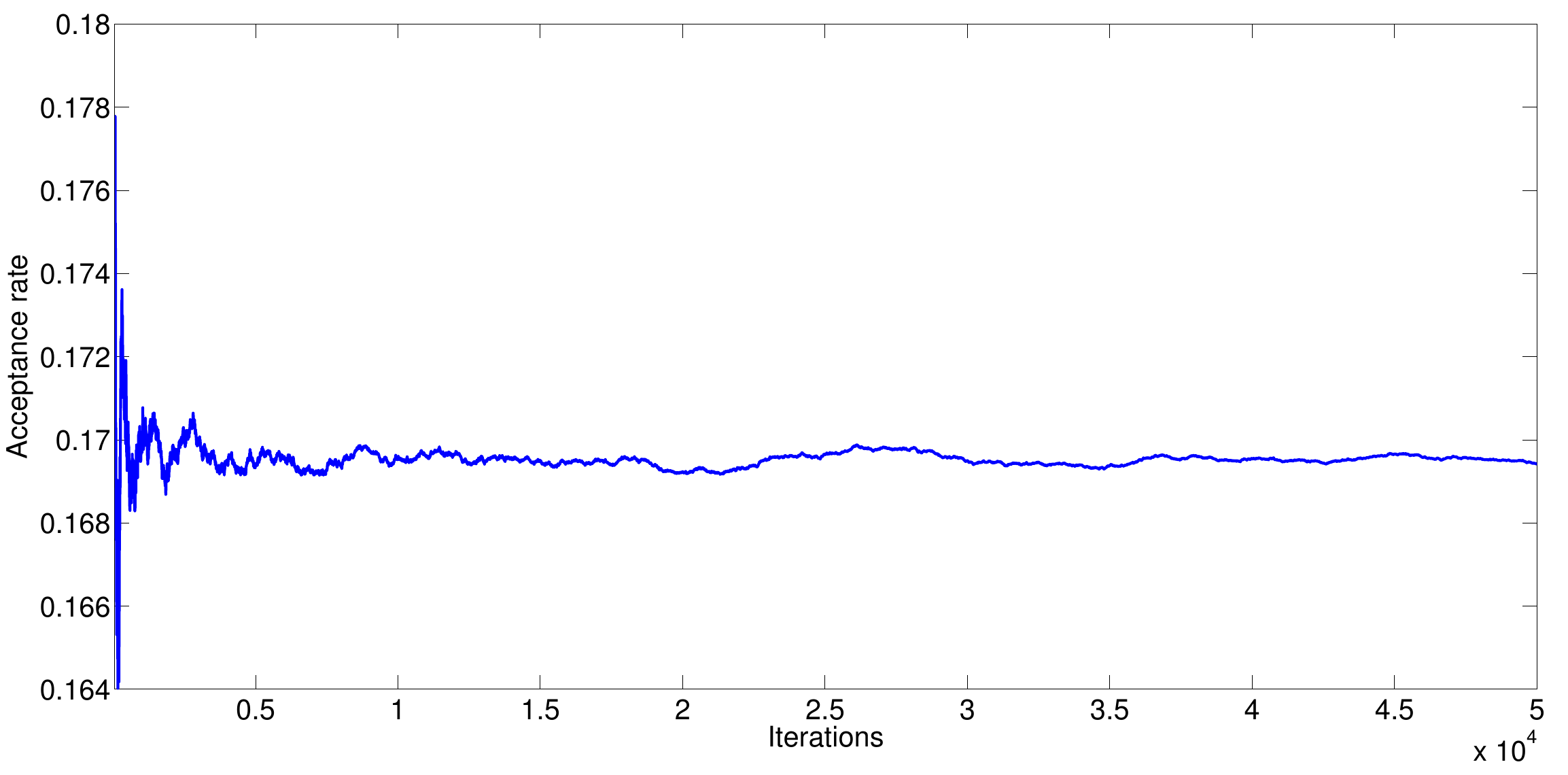}
\end{center}
\caption{(top) True regression vector (stars), mean regression vector (circles) and estimated activation probabilities (bars). (bottom) Mean acceptance rate as a function of the number of iterations.} \label{fig:logistic_x_acc}
\end{figure}

\subsection{Linear regression}
The model for the observations $\Obs \in \Rset^N$ is assumed to be
\[
\Obs = \design \state + \tau^{-1/2}  \noise \eqsp,
\]
where $\design$ is a $N \times \dimstate$ (known) design matrix, $\noise$ is a
Gaussian random vector with i.i.d. standard entries and $\tau$ is the (known)
precision.  The prior on the models is $\omega_m = \theta_\star^{|m|}
(1-\theta_\star)^{p-|m|}$ for some (known) $\theta_\star \in (0,1)$. The
conditional distribution of $X$ given the observations $Y$ and the model $m$ is
given by
\begin{equation*}
  \pi_m\left(x\right) \propto \exp\left( - \frac{\tau}{2 }
    \|Y - \design x \|^2\right)  \times \prod_{\ell=1}^\dimstate \left\{ \left( 1 + \frac{x_\ell^2}{2aK}
    \right)^{ -( a+1/2)} \hspace{-1cm}\indi{m_\ell =1} + \delta_0(x_l) \ 
    \indi{m_\ell = 0} \right\}\eqsp.
\end{equation*}
Such a posterior distribution can be obtained from the following hierarchical
model: \textit{(i)} given $m \in \M$ and positive precisions
$(\vartheta_1,\ldots,\vartheta_{\dimstate})$, the entries $\state =
(\state_1,\ldots,\state_{\dimstate})$ are independent with distribution
\[
\state_k \vert m, \vartheta_1, \ldots, \vartheta_\dimstate \sim   \left\{
  \begin{array}{ll}
\delta_0  & \!\! \text{if $m_k = 0$,} \\
\mathcal{N}(0, 1/\vartheta_k) & \!\! \text{if $m_k =1$.}
  \end{array} \right.
\]
\textit{(ii)} the precision parameters
$\vartheta=(\vartheta_1,\ldots,\vartheta_{\dimstate})$ are i.i.d. with Gamma
distribution $\mathrm{Ga}\left( a, aK\right)$, where $a,K$ are fixed. 

The performance of STMALA is illustrated with the model introduced in
\cite{breiman:1992} and presented in \cite[Section~$8$]{ishwaran:rao:2005}.  We
choose $ N =100$ and $\dimstate=200$. The covariates
$(\design_{\cdot,1},\ldots,\design_{\cdot,\dimstate})$ are sampled from a
Gaussian distribution with $\mathbb{E}[\design_{\cdot,i}] = 0$ and
$\mathbb{E}[\design_{ji}\design_{ki}] = (0.3)^{|j-k|}$; $\tau = 1$. To produce
the observations, we choose the nonzero coefficients of $\state$ in $4$
clusters of $5$ adjacent variables such that, for all $k\in\{1,2,3,4\}$ and all
$j\in\{1,2,3,4,5\}$, $ \state_{50*(k-1)+j} = (-1)^{k+1}\,j^{1/k}$.  Below, this {\em true} value of the regression vector is denoted by $X^\star$. \\
$\theta_\star = 0.1$, $a=2$ and $K = 0.08$. STMALA is run with $\eta = 20$ and
$\gamma=0.35$.

The standard deviation of the RJMCMC proposal is chosen so that STMALA and
RJMCMC have similar acceptance rates (between $15\%$ and $20\%$). 
\begin{figure}[ht!]
\begin{center}
  \includegraphics[width=.75\linewidth]{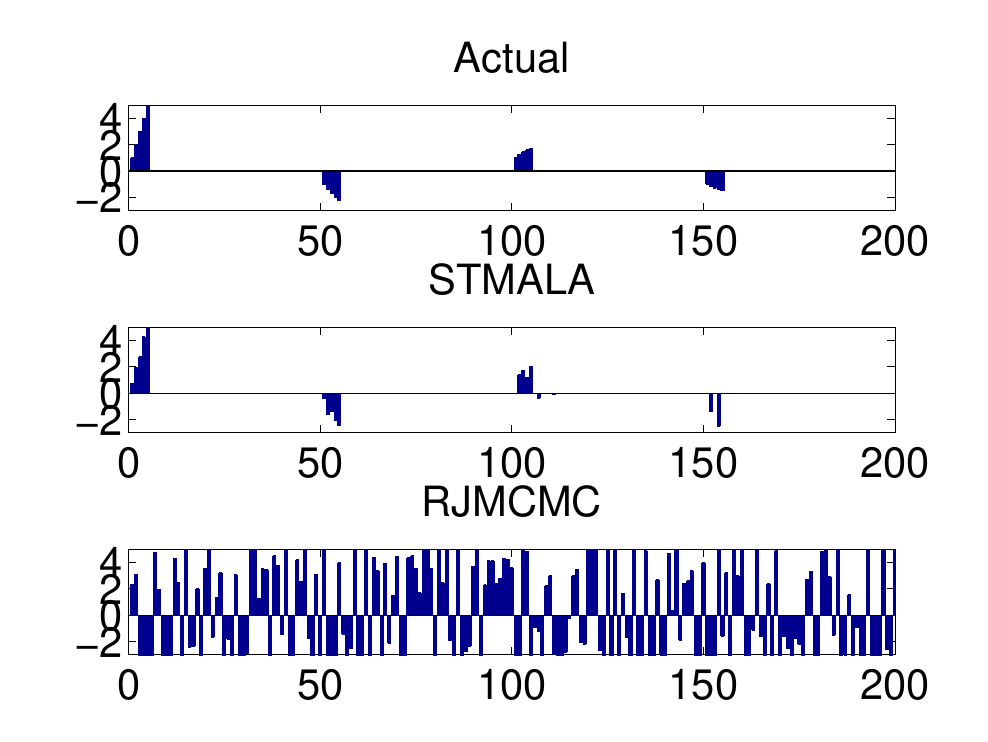}
\end{center}
    \caption{Regression vectors estimated by STMALA and RJMCMC.}
     \label{fig:bre:acc:acf}
\end{figure}
Figure~\ref{fig:bre:acc:acf} shows the true regression vector $\state$ and its
estimates obtained by STMALA and RJMCMC; these estimates $\hat{X}$ are defined
as the posterior mean along a trajectory of length $10^6$ (the first $10 \%$
samples are discarded).  It shows that STMALA provides a sparse estimation
while RJMCMC needs a lot of components to explain the observations. This is
probably because RJMCMC is more or less equivalent to test each model in turn,
which yields slow convergence in high dimensional settings. This slow
convergence is also illustrated in Figure~\ref{fig:bre:nact:bpx}.  $50$
independent trajectories of length $10^6$ are run;
Figure~\ref{fig:bre:nact:bpx} (top) shows the evolution of the mean number
(over the $50$ runs) of active components $|m|$. RJMCMC has not converged after
the $300.000$ iterations while the mean number of active components of STMALA
is stable after few iterations.  Figure~\ref{fig:bre:nact:bpx} (bottom)
displays the boxplots of the estimation of the first component $\state_1$
estimated by STMALA and RJMCMC as a function of the number of iterations.

\begin{figure}[ht!]
\begin{center}
  \includegraphics[width=.7\linewidth]{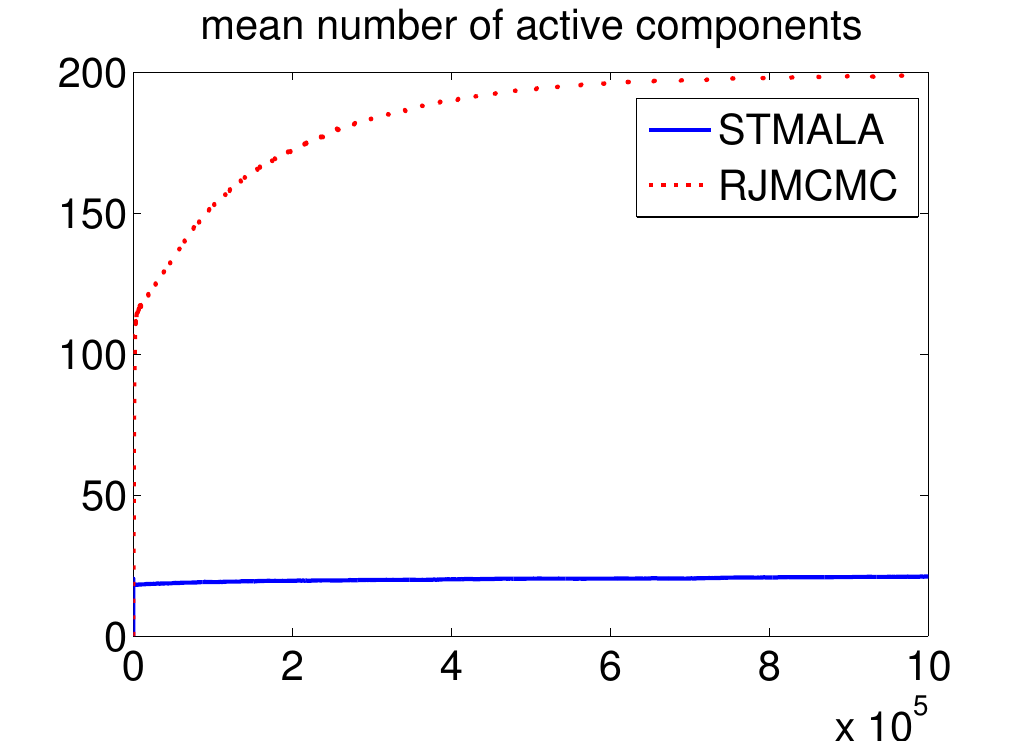}
  \includegraphics[width=.7\linewidth]{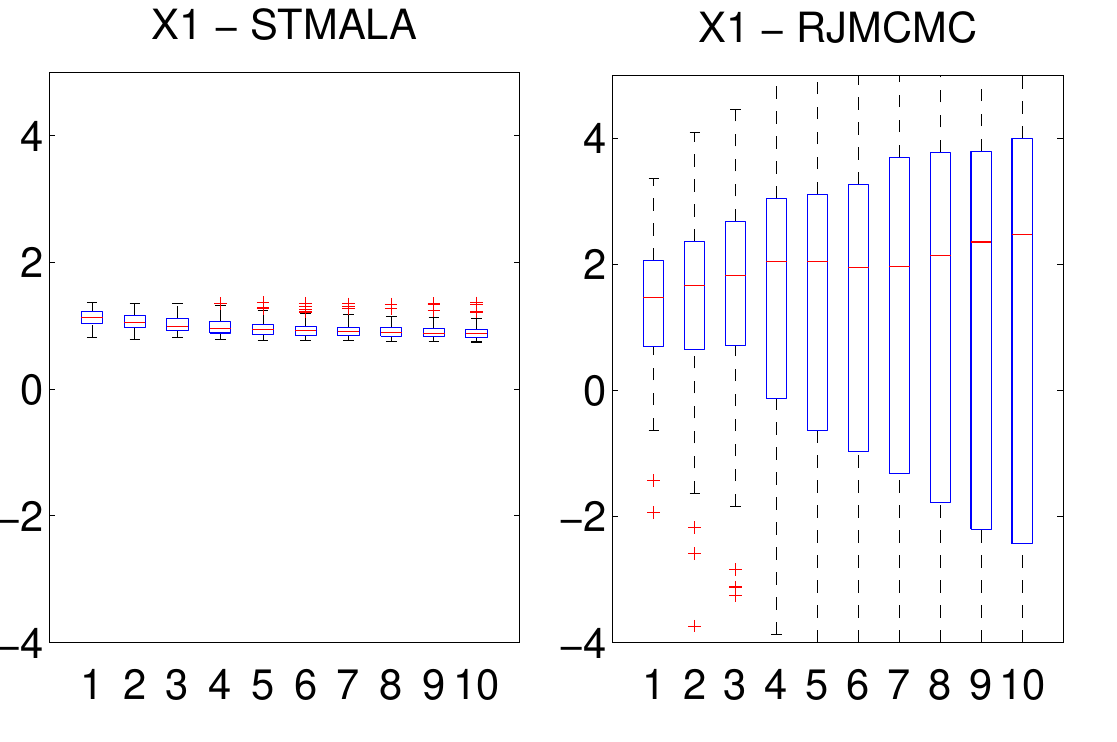}
\end{center}
    \caption{(top) Evolution of the mean number of active components for STMALA and RJMCMC. (bottom) Evolution of the estimation of $\state_1^\star$ (mean over iterations) for STMALA and RJMCMC.}
     \label{fig:bre:nact:bpx}
\end{figure}

Figure~\ref{fig:bre:rvsp:et} (top) shows the signal $\design \hat{\state}$
estimated by STMALA and RJMCMC as a function of the actual emitted signal
$\design\state$ (blue circles), where $\hat{\state}$ is the mean regression
vector over a trajectory. To highlight over fitting effects, a test sample
$Y_{\rm test} = \design_{\rm test} \state^\star + \tau^{-1/2} \noise_{\rm
  test}$, where $\design_{\rm test} \in \Rset^{100 \times 200}$ and
$\noise_{\rm test} \in \Rset^{100}$ are generated exactly as $\design$ and
$\noise$, is also used.  With green circles, $\design_{\rm test} \hat{\state}$
as a function of $\design_{\rm test} \state^\star$ are displayed.  This test
data set is also used to compute a test error, which is given by
\[
\mathcal{E}_{\rm test} \eqdef \frac{\|\design_{\rm test}
  \hat{\state}-\design_{\rm test} \state^\star\|^2}{100}\eqsp.
\]
The evolution of the mean test error $\mathcal{E}_{\rm test}$ over 100
independent runs, is displayed in Figure~\ref{fig:bre:rvsp:et} (bottom).  Both
figures show that RJMCMC is subject to some over fitting, which is not the case
of STMALA.

\begin{figure}[ht!]
\begin{center}
  \includegraphics[width=.65\linewidth]{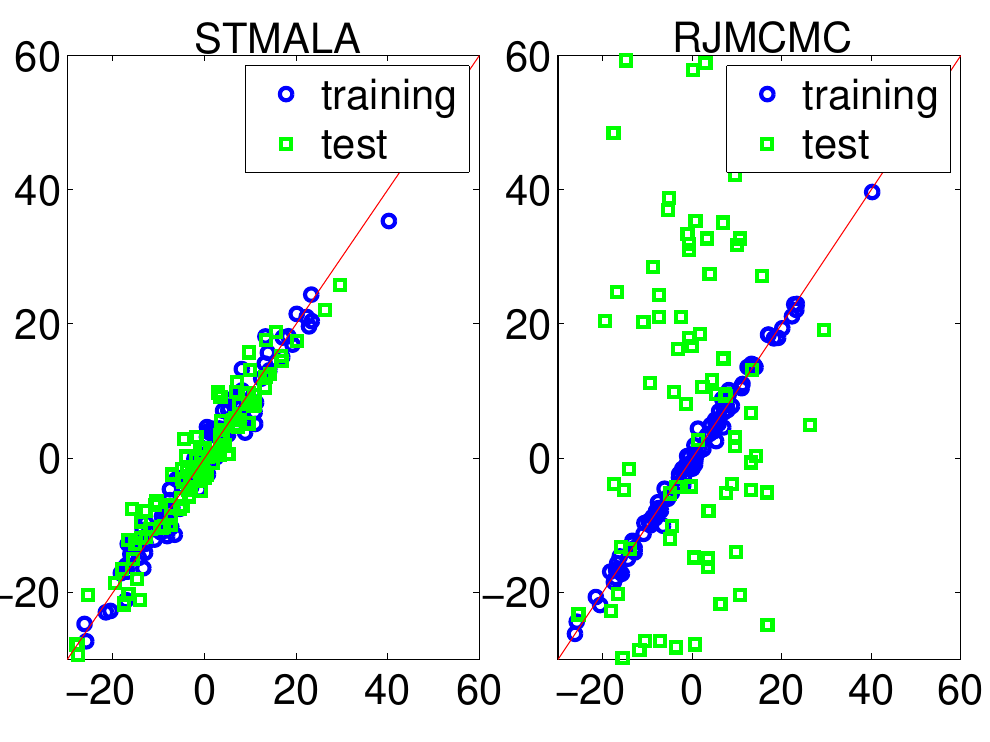}
   \includegraphics[width=.65\linewidth]{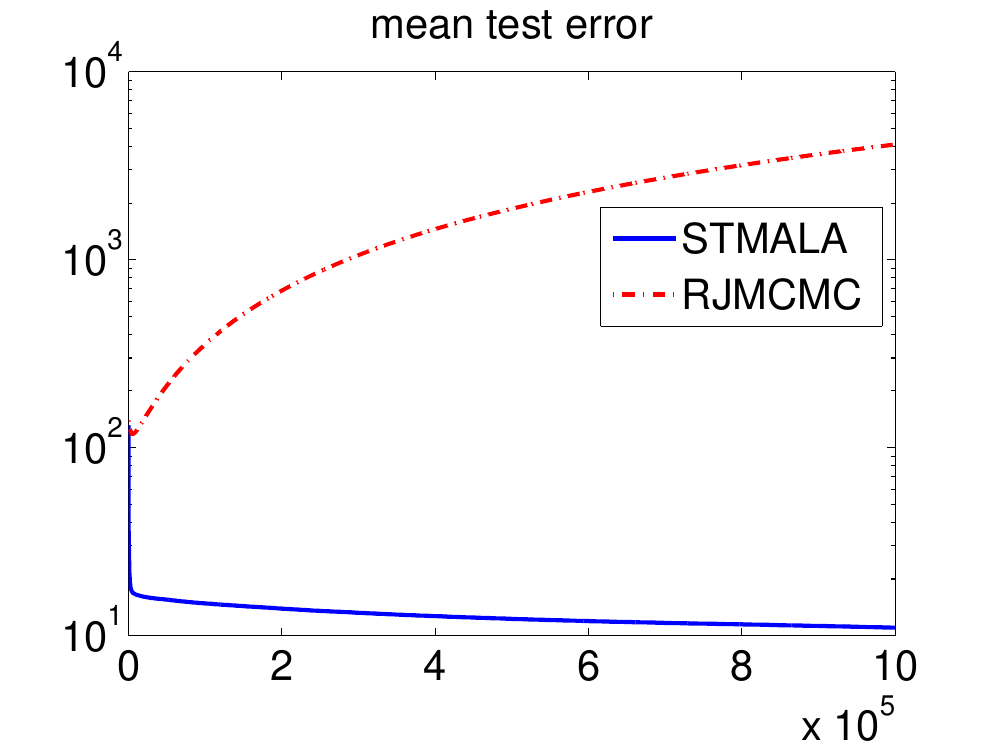}
\end{center}
    \caption{(top) Emitted signal $\design \hat{\state}$ estimated by STMALA and RJMCMC versus actual emitted signal $\design\state$. (bottom) Evolution of the mean test error for RJMCMC and STMALA.}
     \label{fig:bre:rvsp:et}
\end{figure}

\subsection{Regression for spectroscopy data} \label{sec:exp:realdata}
We use the biscuits data set composed of near infrared absorbance spectra of 70
cookies with different water, fat, flour and sugar contents studied in
\cite{brown:fearn:vannucci:2001} and \cite{caron:doucet:2008}. The data are
divided into a training data set containing measurements for $N=39$ cookies, and a test
data set containing measurements for $31$ cookies. The observation model is given by
\[
Y = GX + \tau^{-1/2}E\eqsp,
\]
where $G$ is the design matrix, $X$ is the unknown regression vector and
$E\sim\gauss(0,I)$ is the measurement noise.  Each row of the design matrix
consists of absorbance measurements for $\dimstate=300$ different wavelengths
from $1202$ nm to $2400$ nm with gaps of $4$ nm. We compare the results
obtained by STMALA with those obtained by RJMCMC for the prediction of fat
content. To improve the stability of the algorithm, the columns of the matrix
$G$ containing the measurements are centered and a column with each entry being
equal to one is added.

The parameters of the algorithms are given by $\cnoise = 0.5$, $\eta=15$,
$\gamma = 0.35$ for STMALA. The computations are made over $100$ independent
trajectories of $N_{it}=2.10^6$ iterations, with a burn-in $B=10^5$.  The
design parameters of STMALA and RJMCMC are chosen so that the two algorithms
have similar acceptance-rejection ratios (the final ratios are about $45 \%$
for STMALA and $42 \%$ for RJMCMC).  Figure~\ref{fig:bis:est} shows the
regression vectors $\hat{X}$ obtained by STMALA and RJMCMC, and computed as the
posterior mean along one trajectory (left) and the mean regression vector
estimated by STMALA and RJMCMC over $100$ independent trajectories (right).

The regression vector estimated by STMALA has a spike around $1726$ nm, which
is known to be in a fat absorbance region
(see~\cite{brown:fearn:vannucci:2001,caron:doucet:2008}), in almost all the
trajectories.
\begin{figure}[ht!]
\begin{center}
  \includegraphics[width=.49\linewidth]{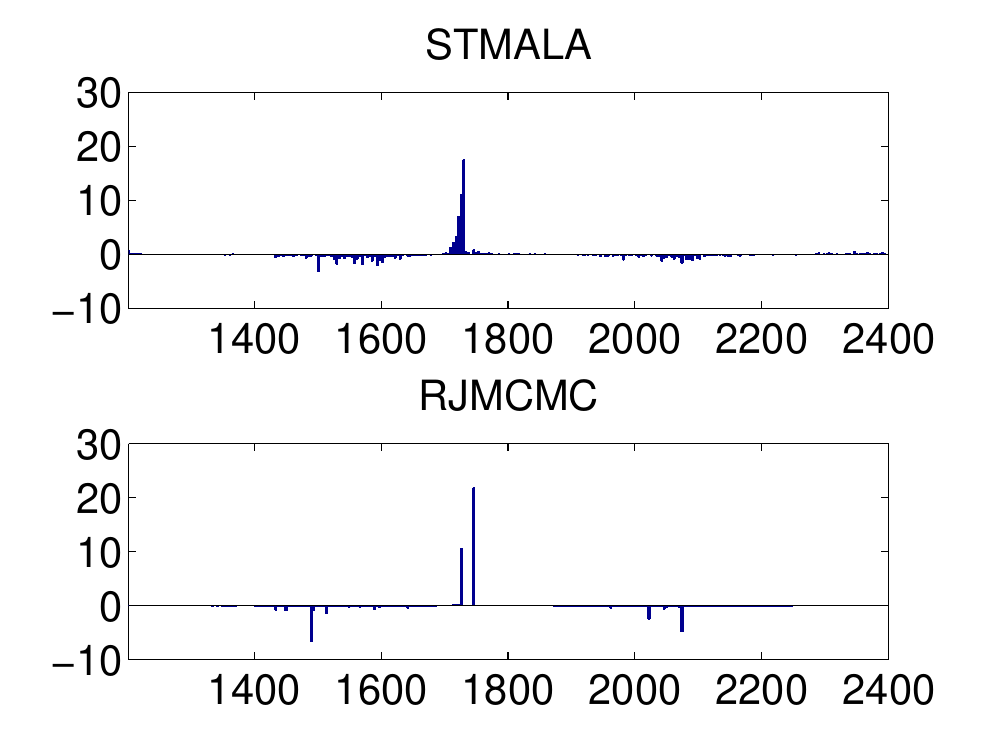}
   \includegraphics[width=.49\linewidth]{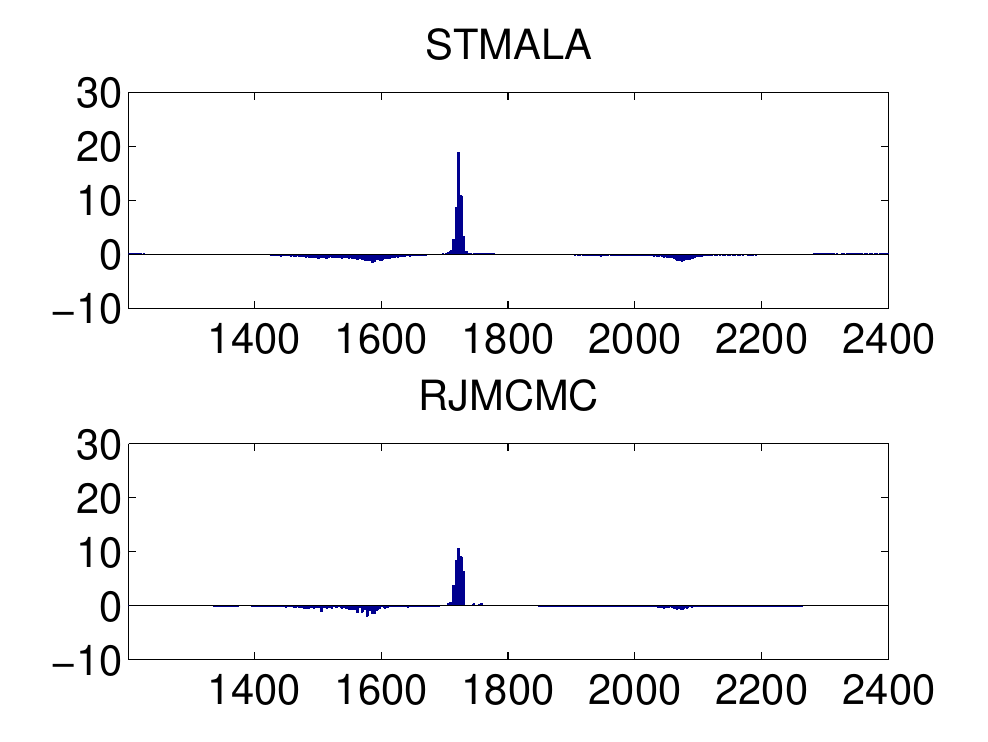}
\end{center}
    \caption{(left) Regression vectors estimated by STMALA and RJMCMC. (right) Mean regression vectors estimated STMALA and RJMCMC over $100$ independent trajectories.}
     \label{fig:bis:est}
\end{figure}

Figure~\ref{fig:bis:bpest} displays the boxplots of the $100$ independent
values of the components of the regression vectors estimated by STMALA and
RJMCMC associated to $9$ wavelengths close to $1726$ nm. It illustrates that
the location of the spike retrieved by RJMCMC is not stable, while STMALA
retrieves a spike centered at $1726$ nm in almost every trajectory.

\begin{figure}[ht!]
\begin{center}
  \includegraphics[height=4.5cm,width=.95\linewidth]{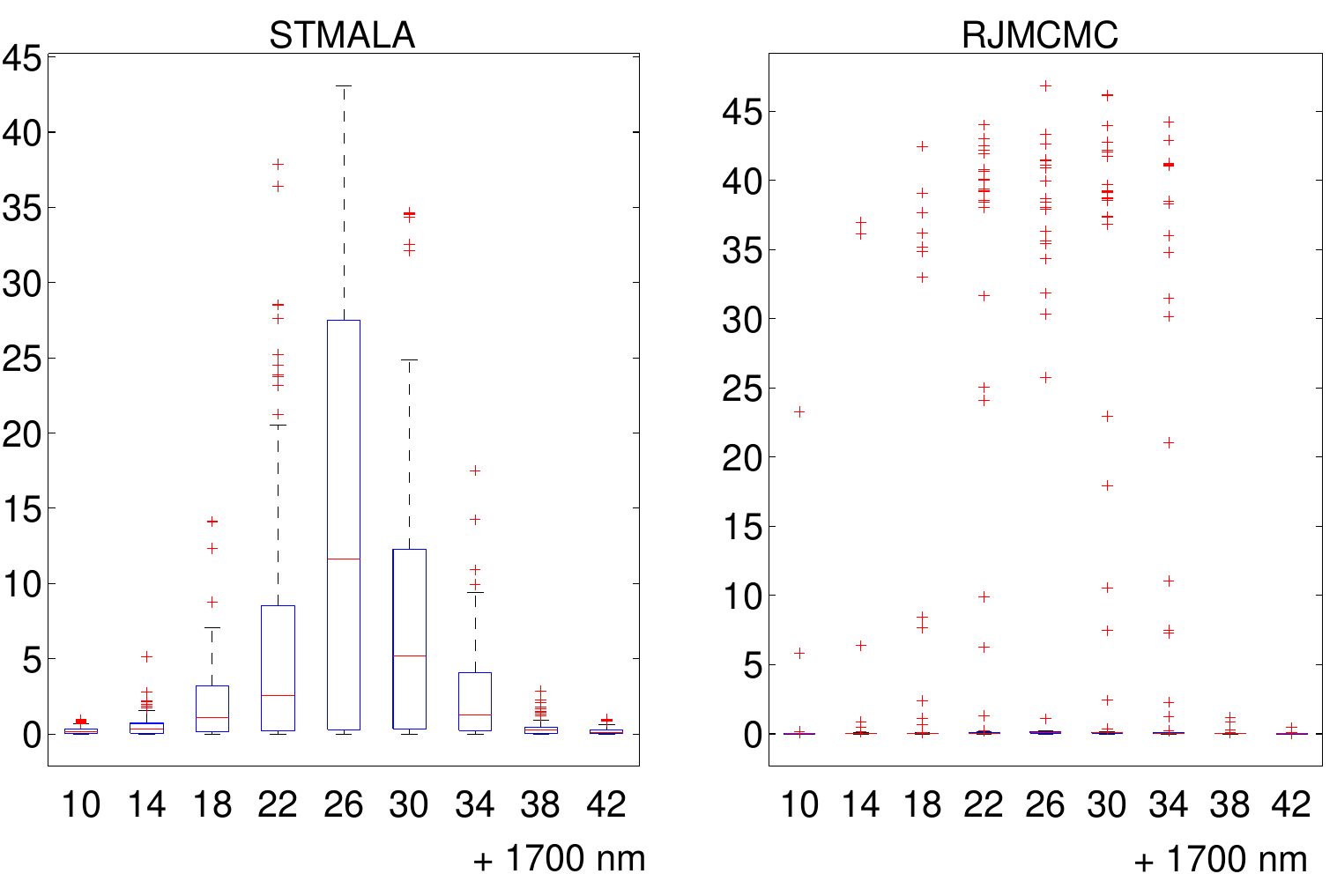}
\end{center}
    \caption{Boxplots of the $100$ independent values of the components of the regression vectors estimated by STMALA and RJMCMC associated with $9$ wavelengths close to $1726$ nm.}
     \label{fig:bis:bpest}
\end{figure}

Figure~\ref{fig:bis:test} (top) shows the estimated emitted signal $\design
\hat{\state}$ obtained by STMALA and RJMCMC as a function of the observations
$\Obs$. In this numerical experiment, STMALA provides better results than
RJMCMC for both the training set and the test set.  This is confirmed by
Figure~\ref{fig:bis:test} (bottom) which displays the evolution of the mean
square error (MSE) on the test dataset, defined by
\begin{align*}
 \textrm{MSE} = \frac{\|G_{\textrm{test}} \hat{X} - Y_{\textrm{test}} \|^2}{31} \eqsp,
\end{align*}
as a function of the number of iterations (mean over 100 independent
trajectories). The mean MSE after $2.10^6$ iterations is about $0.75$ for
STMALA and about $1.6$ times greater for RJMCMC.

\begin{figure}[ht!]
\begin{center}
  \includegraphics[width=.65\linewidth]{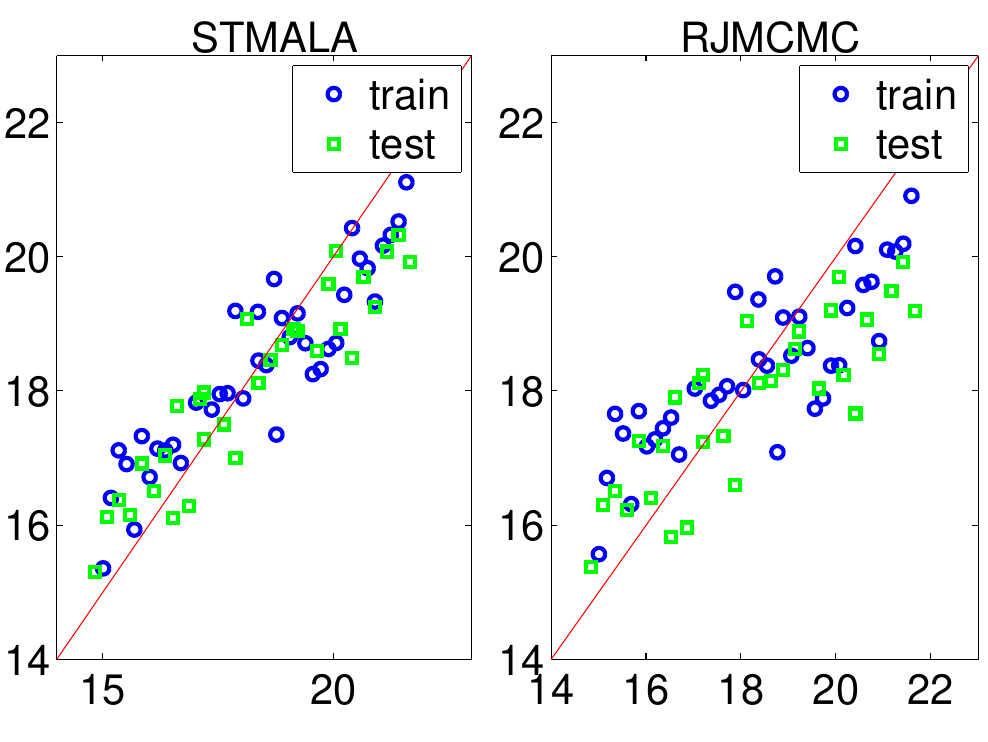}
   \includegraphics[width=.65\linewidth]{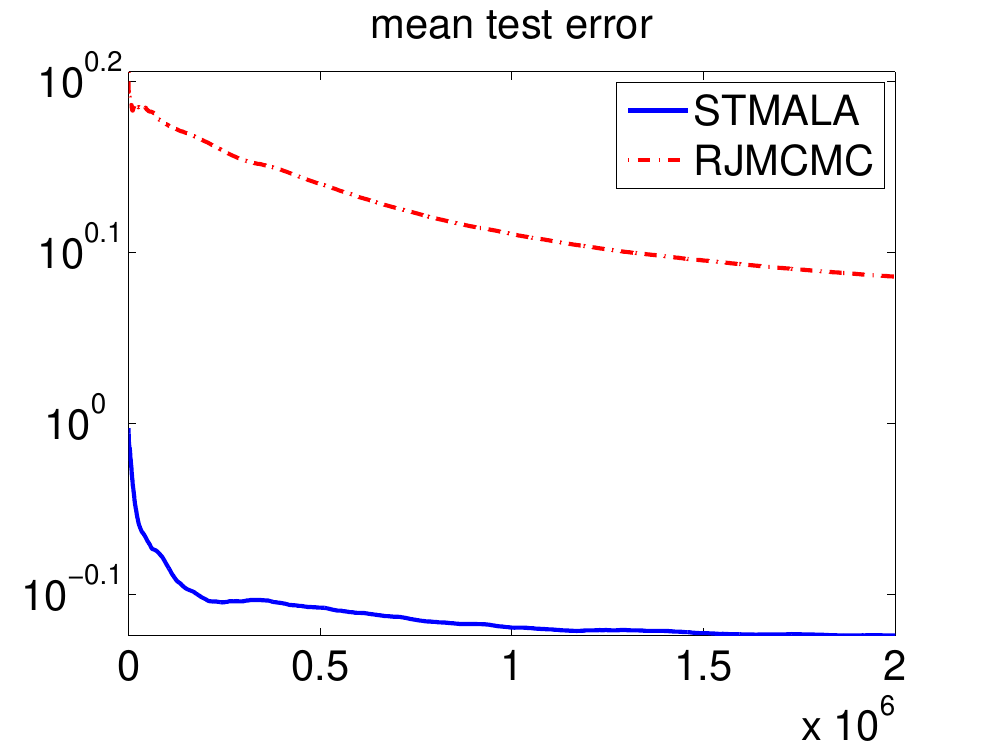}
\end{center}
    \caption{(top) Emitted signal $\design \hat{\state}$ estimated by STMALA and RJMCMC versus the  observations $\Obs$.  (bottom) Evolution of the mean MSE (over $100$ independent trajectories) on the test data set for RJMCMC and STMALA.}
     \label{fig:bis:test}
\end{figure}

\section{Conclusions}
In this paper, we propose a new trans-dimensional MCMC algorithm to
perform Bayesian variable selection in a high-dimensional regression setting.
This algorithm is closely related to \cite{pereyra:2015} but is adapted to
sample models which are exactly sparse in the sense that a certain number of
components are equal to zero. In addition, under fairly weak assumptions, the
STMALA algorithm is shown to be geometrically ergodic.  In the high-dimensional
setting, the STMALA algorithm outperforms the RJMCMC algorithm which is
considered as the state of the art. The performance of the STMALA algorithm
depends on the tuning of a set of parameters: an adaptive version is currently
under investigation. Also, the algorithm has still to be adapted to the ultra
large scale framework, which likely requires additional specific procedures.

\section{Proofs}
\label{PMALA:sec:proofs}
For all $m\in\mathcal{M}$, define $\idxmodel{m}= (\idxmodel{m}_1, \cdots,
\idxmodel{m}_\dimstate)$ as the indices of nonzero coefficients of $m$: $\idxmodel{m}_1 \eqdef \inf\{1\le i\le
\dimstate\;;\;m_i=1\}$ and for $2\le j\le |m|$, $\idxmodel{m}_j \eqdef \inf\{i
> \idxmodel{m}_{j-1}: m_i=1\}$.  Then, for all $x\in\Rset^{|m|}$, let
$\expvec{x}{m}$ be the vector of $\Rset^{\dimstate}$ such that for all $1\le
i\le |m|$, $\expvec{x}{m}_{\idxmodel{m}_i} = x_i$ and for all $i\notin
\{\idxmodel{m}_1,\ldots,\idxmodel{m}_{|m|}\}$, $\expvec{x}{m}_i=0$. For all
$y\in\Rset^p$ and all $m\in\mathcal{M}$, let $\shrinkvec{y}{m}$ be the vector
of $\Rset^{|m|}$ such that for all $1\le i\le |m|$, $(\shrinkvec{y}{m})_i =
y_{\idxmodel{m}_i}$.

\subsection{Proof of Lemma~\ref{STMALA:prox}}
\label{proof:lem:STMAL:prox}
Consider first the case $\dimstate=1$.
We first compute the derivative of $\hstos$ on $\ooint{0, \infty}$ (note that $\hstos$ is symmetric). For any $x \in \ooint{0,\infty}$,
 \begin{equation*}
  \hstos'(x) = \gamma^2 \left[(x^2+4\gamma^2)^{-1/2} +  (x^2+4\gamma^2)^{-1/2} \exp \left(-2 \asinh \left(x/(2 \gamma) \right) \right) \right] \eqsp.
 \end{equation*}
 Using straightforward computations, we get
 \begin{align*}
  \hstos'(x) = (-x + \mathrm{sign}(x) \sqrt{x^2+4 \gamma^2})/2  \eqsp.
 \end{align*}
Set $\psi_u(x) \eqdef \hstos(x) +  \ (x-u)^2/2$.  Since we have $ \psi_{-u}(x) = \psi_u(-x)$, we only have to consider the case when $u \geq 0$.  Hereafter, $u \geq 0$.  It is easily proved that on
$\ooint{0, \infty}$, the derivative $\psi_u'$ is strictly increasing to infinity, and
a solution to the equation $\psi_u'(x) = 0$ exists on $\ooint{0, \infty}$ \iff\ $u >
\gamma$. In this case, this solution is $u-\gamma^2/u$, and
$\psi_u(u-\gamma^2/u) \leq \psi_u(0)$.  When $u \in \coint{0, \gamma}$, $\inf_{x>0}
\psi_u(x) = \psi_u(0)$.  Moreover, it can be proved that $\psi_u'(x)=0$ has no
solution on $\ooint{- \infty, 0}$, and therefore that $\inf_{x<0} \psi_u(x) =
\psi_u(0)$ whatever $u>0$ is. Hence, the minimum is reached at $0$ if $u \in
\coint{0, \gamma}$ and at $u-\gamma^2/u$ if $u >\gamma$.

Consider now the case $\dimstate >1$. Set $x \in \Rset^\dimstate$ of the form $x = r
\xi$ where $r>0$ and $\xi$ is on the unit sphere of $\Rset^\dimstate$. Since the
function $\hstos$ only depends on the radius $r$, the minimum over $\Rset^\dimstate$
of $x \mapsto \hstos(x) + \| x- u \|^2/2$ is reached in the direction
$\xi_\star = u/\|u\|$. Then, finding the minimum in this direction is
equivalent to find the minimum of the function $\psi_u$ on $\Rset^+$, which
yields $r_\star =0$ if $ \|u\| \leq \gamma$ and $r_\star =
(1-\gamma^2/\|u\|^2)$ otherwise. This concludes the proof.

\subsection{Proof of Lemma~\ref{PMALA:lem:propnorm21}}
\label{sec:proofs:pmala:etc}
Let $\varphi$ be a bounded continuous function on $\Rset^{\dimstate}$. Then,
\begin{equation*}
  \PE[\varphi(Z)]  = \left( 2 \pi \sigma^2 \right)^{-\dimstate/2}
  \int_{\Rset^{\dimstate}} \varphi \left(\stog(y) \right) \times \prod_{i=1}^\dimstate\exp
  \left(-\frac{|y_{i}- \mu_{i}(x)|^2}{2 \sigma^2} \right) \rmd
  y\eqsp.
 \end{equation*}
 For $m\in\M$ and $y \in \Rset^{|m|}$, set $\overline{y}=(\overline{y}_1,
 \cdots, \overline{y}_{|m|})$ where $\overline{y}_{i} \eqdef y_{i} \left( 1 -
   \gamma/| y_{i}| \right)$.
Fubini's theorem yields
\begin{multline*}
  \PE[\varphi(Z)] 
= \left( 2 \pi \sigma^2 \right)^{-|m|/2} \sum_{m \in \M} \prod_{i \notin I_m}
  \rho\left(\mu_{i}(x) \right) \times \int_{\Rset^{|m|}} \varphi
  \left(\expvec{\overline{y}}{m}\right) \ \left( \prod_{k=1}^{|m|} \indi{|
      y_{k}| > \gamma} \right) \\
\times \exp \left(-\frac{\|y-\shrinkvec{\mu(x)}{m}\|^2}{2 \sigma^2} \right) \rmd
  y\eqsp.
\end{multline*}
It is sufficient to compute integrals of the form
\begin{equation*}
I(\widetilde{\varphi})= \int_{ \Rset} \widetilde{\varphi} \left( v \left( 1 - \frac{\gamma}{|v|} \right) \right) \, \indi{|v| > \gamma} \times \exp \left(-\frac{|v-\mu|^2}{2 \sigma^2} \right)   \rmd v \eqsp,
\end{equation*}
for a generic function $\widetilde{\varphi}$. Consider the change of variable
$\Rset\setminus [-\gamma,\gamma] \to \Rset^{\star}$: $z = v \left( 1-
  \frac{\gamma}{|v|}\right)$. Note that $|z| = |v|-\gamma$ and $v = \psi(z)$,
where for any $z \in \Rset^{\star}$, $\psi(z) \eqdef (1+ \gamma/|z|)z$.  Then,
\[
I(\widetilde{\varphi})= \int_{\Rset^{\star}}   \widetilde{\varphi} \left(v \right) \, \exp\left(-\frac{|\psi(v)-\mu|^2}{2 \sigma^2} \right)\rmd v\eqsp.
\]
This concludes the proof for $\stog$. The proof for $\stos$ follows the same lines as the proof of
Lemma~\ref{PMALA:lem:propnorm21}, with the function $\psi$ replaced by $
\widetilde{\psi}(z) = g\left(\gamma^2/\|z\|^2 \right) \ z$.

\subsection{Proof of Theorem~\ref{th:erggeo}}
\label{sec:proofs:erg}
For ease of notations, we denote by $q$ the proposal distribution.
Lemma~\ref{PMALA:lem:propnorm21} shows that for any $m \in \M$ and $y \in S_m$
\begin{equation}
  \label{eq:local:q}
q(x,y)=   \prod_{i \notin I_m} \rho\left( \mu_i(x)
\right)  \ \prod_{i \in I_m}  f\left( \mu_i(x), y_i\right) \eqsp,
\end{equation}
where $\rho$ and $f$ are given by Lemma~\ref{PMALA:lem:propnorm21} and $\mu(x)
= (\mu_1(x), \cdots, \mu_\dimstate(x))$ is given by \eqref{eq:local:mu}.   We
start with a preliminary lemma which will be fundamental for the proofs since
it allows to compare the proposal distribution $q$ to Gaussian proposals. Denote by $g_\sigma$ the one-dimensional centered Gaussian density with standard deviation $\sigma$.
\begin{lemma} \label{lemme:bgaus}
There exist $\cigaus$, $\cgaus$, $\vigaus$ and $\vgaus$ such that For any $x,y \in \Rset^\dimstate$ and any $1 \leq i \leq \dimstate$,
 \begin{equation*}
  \cigaus \ g_{\vigaus}(y_i - x_i) \leq f(\mu_i(x), y_i) \leq \cgaus \ g_{\vgaus}(y_i - x_i) \eqsp,
 \end{equation*}
\end{lemma}

\begin{proof}
  Assume first that $\sto = \stog$. Let $x,y \in \Rset^\dimstate$ and $ i \in \{1, \cdots, \dimstate\}$. By
  definition of $\mu$ (see \eqref{eq:local:mu}), we have $ \left| \mu_i(x)- x_i
  \right| \leq \| \mu(x) - x\| \leq D \sigma^2/2$.  Thus,
\[
\left| y_i - x_i \right|
\leq \left| y_i + \gamma \sign(y_i) - \mu_i(x) \right| + \gamma +
  \frac{D \sigma^2}{2} \eqsp,
\]
which implies $\left| y_i + \gamma \sign(y_i) - \mu_i(x) \right|^2 \geq
\frac{1}{2} \left| y_i - x_i \right|^2 - \left(\gamma + D \sigma^2/2\right)^2$.
Similarly, $ \left| y_i + \gamma \sign(y_i) - \mu_i(x) \right|^2 \leq 2 \left|
  y_i-x_i \right|^2 + 2 \left(\gamma + D \sigma^2/2\right)^2$.  
  
Assume now that $\sto = \stos$ and let $x,y \in \Rset^\dimstate$ and $ i \in \{1, \cdots, \dimstate\}$. First,
\begin{equation*}
g(\gamma^2|y|^{-2})\tilde{g}(\gamma^2|y|^{-2})  = (1+1/\sqrt{1+4\gamma^2|y|^{-2}})/2\eqsp,
\end{equation*}
which yields $1/2\le g(\gamma^2|y|^{-2})\tilde{g}(\gamma^2|y|^{-2})\le 1$.
Furthermore,
\begin{align*}
|g(\gamma^2|y|^{-2}) y-\mu(x)| &\le |g(\gamma^2 |y|^{-2}) y -y|+|y-x|+|x-\mu(x)|\le \gamma+|y-x|+D \sigma^2/2\eqsp,
\end{align*}
On the other hand,
\begin{align*}
|y-x| \le |g(\gamma^2|y|^{-2})y-\mu(x)| + |g(\gamma^2|y|^{-2})y-y|+|x-\mu(x)| \le |g(\gamma^2|y|^{-2})y-\mu(x)| + \gamma + D \sigma^2/2\eqsp.
\end{align*}
\end{proof}
\begin{corollary}
\label{lemme:bgaus:maj:q} For any $x \in \Rset^\dimstate$ and $y \in S_m$,
  $q(x,y) \leq \cgaus^{|m|} \prod_{i \in I_{m}} g_{\vgaus}(y_i-x_i)$.
  Therefore, there exists a constant $C>0$ such that for any $x,y \in \Rset^\dimstate$,
  $q(x,y) \leq C$.
\end{corollary}
The proof of Theorem~\ref{th:erggeo} also requires a lower bound on the probability that a component of the proposed point will be set to zero. Such a bound is given in Lemma~\ref{lem:minoration:probap}.
\begin{lemma}
\label{lem:minoration:probap} Let $\rho$ and $\mu$  be given by Lemma~\ref{PMALA:lem:propnorm21} and  \eqref{eq:local:mu}. It holds
\[
\inf_{m \in \M} \inf_{z \in S_m} \inf_{i \not \in I_m} \rho(\mu_i(z)) > 0 \eqsp.
\]
\end{lemma}
\begin{proof}
For $i \not \in I_m$,  by \eqref{eq:local:mu}, $| \mu_i(z) | \leq D
  \sigma^2 /2$. Hence, there exists a constant
  $C>0$ such that
\begin{align} \label{eq:min:p21}
  \inf_{z \in \Rset^\dimstate} \min_{i \not \in I_m} \PP(|\mu_i(z) + \sigma \xi| \leq
  \gamma) \geq C \eqsp,
\end{align}
where $\xi \sim \mathcal{N}(0,1)$.
\end{proof}

\begin{proposition}\label{prop:small}
  \begin{enumerate}[(i)]
  \item \label{prop:small:set} Let $C$ be a Borel set of $\Rset^\dimstate$ such
    that for any $m \in \M$, $C \cap S_m$ is a compact set of $S_m$. Then, $C$
    is a one-small set for the kernel $\trunckernel_\sto$: there exists a positive
    measure $\tilde\nu$ on $\Rset^\dimstate$ such that $\trunckernel_\sto(x,A) \geq
    \tilde \nu(A) \un_{C}(x)$.
  \item \label{prop:small:irreducibility} The Markov kernel $\trunckernel_\sto$ is
    psi-irreducible and aperiodic.
  \end{enumerate}
\end{proposition}
\begin{proof}
  For notation simplicity, we drop the dependency in $\sto$ {\em
    \eqref{prop:small:set}}. We set $\nu = \sum_{m \in \M} \nu_m$.  Let $C$ and
  $K$ be two Borel sets of $\Rset^\dimstate$ such that $\nu(K) >0$ and for any
  $m \in \M$, $C \cap S_m$ and $K \cap S_m$ are compact subsets of $S_m$.
  Since $\Rset^\dimstate = \bigcup_{m \in \M} S_m$, we have
\begin{align*}
\inf \limits_{ x \in C} \trunckernel(x,A) = \inf \limits_{m \in \M} \inf \limits_{ x \in C \cap S_m} \trunckernel( x,A) \eqsp,
\end{align*}
so that it is enough to establish a minorization on the kernel for any $x \in C
\cap S_{m_\star}$ whatever $m_\star \in \M$. Let $m_\star \in \M$.  By
definition of $\trunckernel$, $q$ (see (\ref{eq:local:q})) and $\nu$
\[
\trunckernel( x,A)  \geq \int_{A \cap K} \alpha( x, y) q( x, y) \dom( y)
\]
where, for any $x \in S_{m_\star}$ and $y \in S_m$, we have
\[
q(x,y)= \prod_{i \notin I_{m}} \rho ( \mu_i(x) ) \prod_{i \in I_{m}}
f_\sto\left( \mu_i(x), y_i\right)\eqsp.
\]
The latter inequality implies
\begin{equation*}
\trunckernel(x,A) \geq \sum_{m \in \M} \cigaus^{|m|} \prod_{i \notin I_m} \rho(\mu_i(x)) \times \int_{A \cap K \cap S_m} \alpha( x,y) \prod_{i \in I_m} g_{\vigaus}(x_i-y_i) \rmd y_i \eqsp,
\end{equation*}
where the last inequality follows from
Lemma~\ref{lemme:bgaus}.  For any $x \in
S_{m_\star}$ and $y \in S_m$, we have
\[
\alpha_\sto(x,y) = 1 \wedge \frac{\omega_{m} \pi_{m}(y) q(y,x)}{\omega_{m_\star}
\pi_{m_\star}(x) q(x,y)}\eqsp.
\]
There exists a compact set of $\Rset$ such that for any $x \in C \cap
S_{m_\star}$ and $y \in K \cap S_m$, $\mu_i(x)$ and $\mu_i(y)$ are in this
compact for any $i$. Hence,
A\ref{hyp:reg:pi}\eqref{hyp:reg:pi:pos}-\eqref{hyp:reg:pi:cont} and
Lemmas~\ref{lemme:bgaus} and
\ref{lem:minoration:probap} imply that there exists $\varepsilon_m >0$ such
that for any $x \in C \cap S_{m_\star}$ and $y \in K \cap S_m$,
\[
\alpha_\sto(x,y) \geq \varepsilon_m \eqsp, \qquad \inf_{i \in I_m}
g_{\vigaus}(x_i-y_i) \geq \varepsilon_m \eqsp.
\]
This yields for any $x \in C \cap S_{m_\star}$, $\trunckernel( x,A) \geq
\left(\inf_{m \in \M} \varepsilon_m \right) \int_{A} \un_{K} ( y) \rmd \nu(y)$, thus
concluding the proof.

{\em \eqref{prop:small:irreducibility}}: By~\cite[Lemma
1.1]{mengersen:tweedie:1996}, the Markov chain $\left(\state^n\right)_{n\ge 0}$
is psi-irreducible since for any $x,y \in \Rset^\dimstate$, $q(x,y)>0$ as a
consequence of Lemma~\ref{lemme:bgaus} and
strongly aperiodic since by Proposition~\ref{prop:small}\eqref{prop:small:set}
it possesses an accessible $1$-small set.
\end{proof}

For any measurable function $f: \Rset^p \to \Rset^+$, $\trunckernel f:
\Rset^\dimstate \to \Rset^+$ denotes $\trunckernel f(x)= \int
\trunckernel(x,\rmd z) f(z)$. Fix $\constV \in (0,1)$ and set $V:
\Rset^\dimstate \to \coint{1, \infty}$, $x \mapsto c_\constV \pi^{-\constV}(x)$.
Define the possible rejection region $R(x)$ by
\begin{align*}
R(x) \eqdef \{y \in \Rset^\dimstate: \pi(x) q(x,y) > \pi(y) q(y,x) \} \eqsp.
\end{align*}
We have
 \begin{equation} \label{eq:Majoration:Drift}
   \frac{\trunckernel V(x)}{V(x)}
\leq \sum_{m \in \M}  \left\{ T_m(x) + \int_{R(x) \cap S_m} q(x,y) \dom_m(y) \right\} \eqsp,
\end{equation}
where
\begin{align} \label{eq:def:tm}
  T_m(x) & \eqdef \int_{\Rset^{|m|}} \alpha_\sto(x, \expvec{z}{m}) \,
  \frac{\pi^{-\beta}(\expvec{z}{m})}{\pi^{-\beta}(x)} \, q(x,\expvec{z}{m}) \, \rmd z
  \eqsp.
\end{align}

\begin{lemma} \label{lemme:tm}
  For any $m \in \M$, $\limsup \limits_{\|x\| \to \infty} T_m(x) = 0$.
\end{lemma}
\begin{proof}
  The proof is adapted from~\cite{jarner:hansen:2000} and \cite{atchade:2006}.
  Let $m \in \M$ be fixed.  Define
  \begin{align*}
    \calB_m(x,a) & \eqdef \{z \in \Rset^{|m|}, \|z-\shrinkvec{x}{m}\| \leq a \} \eqsp,\\
    \calC_{m}(x) &\eqdef \{ z \in \Rset^{|m|}, \pi(\expvec{z}{m})=\pi(x) \} \eqsp, \\
    \calC_{m}(x,u) & \eqdef \{z + s n(z), |s| \leq u, z \in \calC_{m}(x) \}
    \eqsp, \\
    R_m(x) & \eqdef \Rset^{|m|} \setminus A_m(x) \eqsp,
  \end{align*}
where
\begin{equation*}
    A_m(x) \eqdef \{z \in \Rset^{|m|}, \pi(\expvec{z}{m}) q(\expvec{z}{m},x) \geq \pi(x) q(x,\expvec{z}{m}) \} \eqsp.
\end{equation*}
We decompose as follows
\[
T_m(x) \leq T_{m,1}(x,a) +\sum_{j=2}^4 T_{m,j}(x,a,u)\eqsp,
\]
where
\begin{align*}
 T_{m,1}(x,a) & \eqdef  \int_{\calB_m^c(x,a)} \alpha(x,\expvec{z}{m}) \frac{\pi^{-\beta}(\expvec{z}{m})}{\pi^{-\beta}(x)} q(x,\expvec{z}{m})  \rmd z \eqsp, \\
 T_{m,2}(x,a,u) & \eqdef  \int_{ \calB_m(x,a) \cap \calC_{m}(x,u)} \hspace{-1cm}\alpha(x,\expvec{z}{m}) \frac{\pi^{-\beta}(\expvec{z}{m})}{\pi^{-\beta}(x)} q(x,\expvec{z}{m})  \rmd z \eqsp, \\
 T_{m,3}(x,a,u) &\eqdef \int_{A_m(x)\cap \calB_m(x,a) \cap \calC^c_{m}(x,u)} \hspace{-.2cm}\frac{\pi^{-\beta}(\expvec{z}{m})}{\pi^{-\beta}(x)} q(x,\expvec{z}{m})  \rmd z \eqsp, \\
 T_{m,4}(x,a,u) &\eqdef \int_{R_m(x)\cap \calB_m(x,a) \cap \calC^c_{m}(x,u)}
  \hspace{-.2cm}\frac{\pi^{1-\beta}(\expvec{z}{m})}{\pi^{1-\beta}(x)} q(\expvec{z}{m},x) \rmd
  z\eqsp.
\end{align*}
We prove that we may choose the constant $C>0$ large enough so  that for any $\epsilon >0$ there exists
$M>0$ such that $\sup_{\|x\| \geq M} T_m(x) \leq C \epsilon$. Since $\epsilon$
is arbitrarily small, this yields the lemma. Note that for any $z \in
\Rset^{|m|}$,
\begin{equation}
  \label{eq:bound:partout}
  \alpha_\sto(x,\expvec{z}{m}) \frac{\pi^{-\beta}(\expvec{z}{m})}{\pi^{-\beta}(x)} \leq \left( \frac{q(\expvec{z}{m},x)}{q(x,\expvec{z}{m})}\right)^\beta \eqsp.
\end{equation}
\paragraph{Control of $T_{m,1}$}
By (\ref{eq:bound:partout}), $ T_{m,1}(x,a)\leq \int_{\calB_m^c(x,a)}
q(x,\expvec{z}{m})^{1-\beta} q(\expvec{z}{m},x)^{\beta} \rmd z$. By (\ref{eq:local:q})
and Lemma~\ref{lemme:bgaus}, there exists a
constant $C>0$ such that
\begin{multline*}
 T_{m,1}(x,a)  \leq C \cgaus^{|m|(1-\beta)}  \times \int_{\calB_m^c(x,a)} \prod_{i} g_{\vgaus}\left( ( \shrinkvec{x}{m} )_i - y_i\right)^{1-\beta}  \rmd y_i \\
 \leq C \cgaus^{|m| (1-\beta)} \int_{\calB_m^c(0,a)} \prod_{i} g_{\vgaus}(y_i)^{1-\beta}  \rmd y_i \eqsp.
\end{multline*}
Therefore, for any $\epsilon >0$, there exists $a>0$ such that $\sup_{x \in
  \Rset^\dimstate} T_{m,1}(x,a) \leq \epsilon$.

\paragraph{Control of $T_{m,2}$}
By (\ref{eq:bound:partout}), $ T_{m,2}(x,a,u) \leq \int_{\calB_m(x,a) \cap
  \calC_{m}(x,u)} q(x,\expvec{z}{m})^{1-\beta} q(\expvec{z}{m},x)^{\beta} \rmd
z$.  By A\ref{hyp:superexp:pi}, the Lebesgue measure of $\mathcal{B}_m(x,a)
\cap \calC_{m}(x,u)$ can be made arbitrarily small - independently of $x \in
\Rset^\dimstate$ - when $u$ is small enough (see~\cite[Proof of Theorem
4.1]{jarner:hansen:2000} for details). Therefore, since $q$ is bounded (see
Corollary~\ref{lemme:bgaus:maj:q}), for any $\epsilon>0$, there
exists $u>0$ such that for any $a>0$: $ \sup_{x \in \Rset^\dimstate}
T_{m,2}(x,a,u) \leq \epsilon$.

\paragraph{Control of $T_{m,3}$}
Set $d_r(u) \eqdef \sup_{\|x\| \geq r} \pi(x + u \, n(x))/\pi(x)$.  By
A\ref{hyp:superexp:pi}, for any $\epsilon, u>0$, there exists $r>0$ large
enough so that $\left(d_{r-u}(u) \right)^{1-\beta} \vee
\left(d_r(u)\right)^{1-\beta}\leq \epsilon$.  By A\ref{hyp:reg:pi}, $\sup
\limits_{z \in \calB_m(0,r)} \pi(\expvec{z}{m})^{-\beta} < \infty$, so that
by corollary~\ref{lemme:bgaus:maj:q}
\begin{align*}
  \sup_{x \in \Rset^\dimstate} \int_{\mathcal{I}_m(x,a,u,r)  }
  q(x,\expvec{z}{m}) \pi^{-\beta}(\expvec{z}{m}) \rmd z < \infty \eqsp,
\end{align*}
where
\begin{equation*}
\mathcal{I}_m(x,a,u,r) \eqdef A_m(x) \cap \calB_m(x,a)
 \cap \calC^c_{m}(x,u) \cap  \calB_m(0,r) \eqsp.
\end{equation*}
A\ref{hyp:reg:pi}\eqref{hyp:reg:pi:lim} implies that
\[
\limsup_{\|x \| \to \infty} \int_{\mathcal{I}_m(x,a,u,r) }
\frac{\pi^{-\beta}(\expvec{z}{m})}{\pi^{-\beta}(x)} q(x,\expvec{z}{m}) \rmd z =
0 \eqsp.
\]
Moreover, by definition of $A_m(x)$, for any $z \in A_m(x)$ it holds
\[
\frac{\pi^{-\beta}(\expvec{z}{m})}{\pi^{-\beta}(x)} q(x,\expvec{z}{m}) \leq
\frac{\pi^{1-\beta}(\expvec{z}{m})}{\pi^{1-\beta}(x)} q(\expvec{z}{m},x) \eqsp;
\]
by corollary~\ref{lemme:bgaus:maj:q}, there exists a
constant $C$ such that for any $x \in \Rset^\dimstate$ and $z \in A_m(x)$
\begin{equation*}
\frac{\pi^{-\beta}(\expvec{z}{m})}{\pi^{-\beta}(x)} q(x,\expvec{z}{m}) 
\leq C \left(
  \frac{\pi^{-\beta}(\expvec{z}{m})}{\pi^{-\beta}(x)} \wedge
  \frac{\pi^{1-\beta}(\expvec{z}{m})}{\pi^{1-\beta}(x)} \right)\eqsp.
\end{equation*}
This yields there exists $C_\star$ such that for any $a,u,r>0$,
\begin{multline*}
  \int_{A_m(x)\cap \calB_m(x,a) \cap \mathcal{J}_m(x,u,r)}
  \hspace{-.2cm}\frac{\pi^{-\beta}(\expvec{z}{m})}{\pi^{-\beta}(x)} q(x,\expvec{z}{m})\rmd z  \\
  \leq C_\star \, \left( \sup_{z \in \mathcal{J}_m(x,u,r)}
    \frac{\pi^{\beta}(x)}{\pi^{\beta}(\expvec{z}{m})} \wedge\sup_{z \in
      \mathcal{J}_m(x,u,r)}
    \frac{\pi^{1-\beta}(\expvec{z}{m})}{\pi^{1-\beta}(x)} \right) \eqsp,
\end{multline*}
where
\[
\mathcal{J}_m(x,u,r) \eqdef \calC^c_{m}(x,u)\cap \calB^c_m(0,r) \eqsp.
\]
Let $z \in \calC_m^c(x,u) \cap \{z: \pi(\expvec{z}{m}) < \pi(x) \}$. By
A\ref{hyp:reg:pi}\eqref{hyp:reg:pi:cont}, $h: s \mapsto \pi(\expvec{z}{m} - s \
n(\expvec{z}{m})) - \pi(x)$ is continuous, and by definition of
$\calC_m^c(x,u)$, $h(s) \neq 0$ for any $0 \leq s \leq u$.  Since $h(0)<0$ (we
assumed that $\pi(\expvec{z}{m}) < \pi(x)$), this implies that $h(u)<0$ i.e.
$\pi(\expvec{z}{m} - s n(\expvec{z}{m})) \leq \pi(x)$.  Then,
\begin{equation*}
\sup_{z \in \calC^c_m(x,u) \cap \calB^c_m(0,r)} \frac{\pi
    (\expvec{z}{m})}{\pi(x)}  \leq \frac{\pi(\expvec{z}{m})}{\pi(\expvec{z}{m} - s n(\expvec{z}{m}))}
  \leq d_{r-u}(u) \eqsp.
\end{equation*}
If $z\in \calC_m^c(x,u) \cap \{z: \pi(\expvec{z}{m}) \geq \pi(x) \}$, we obtain
similarly that $ \pi(x)/\pi(\expvec{z}{m}) \leq d_{r}(u)$. Hence, we
established that
\begin{align*}
  \sup_{z \in \calC^c_m(x,u) \cap \calB^c_m(0,r)}
  \frac{\pi(\expvec{z}{m})}{\pi(x)} \leq d_{r}(u) \vee d_{r-u}(u)\eqsp.
\end{align*}
As a conclusion, there exists $C_\star>0$ and for any $\epsilon,a,u>0$, there
exists $M>0$ such that $\sup_{\|x\| \geq M} T_{m,3}(x,a,u) \leq C_\star
\epsilon$.

\paragraph{Control of $T_{m,4}$}
Following the same lines as for the control of $T_{m,3}(x,a,u)$, it can be
shown that there exists $C_\star>0$ and for any $\epsilon,a,u>0$, there exists
$M>0$ such that $\sup_{\|x\| \geq M} T_{m,4}(x,a,u) \leq C_\star \epsilon$.
\end{proof}

\begin{lemma}
  \label{lemme:cone}
  Let $u,b,\epsilon,R$ be given by A\ref{hyp:cone:pi} and $W_m(x)$ be defined
  by \eqref{eq:definition:cone}. There exists $r >R$ such that for any $m \in
  \M$ and $x \in S_m \cap \{\|x \| \geq r\}$, $W_m(x) \subset \{y \in S_m,
  \alpha_\sto(x, y) =1 \}$.
\end{lemma}
\begin{proof}
  The proof is adapted from~\cite{jarner:hansen:2000}. Let $m \in \M$ and
  $x \in S_m$ such that $\| x \| \geq r$ for some $r>R$ to be fixed later (the
  constant $R$ is given by A\ref{hyp:cone:pi}).  We first prove that there
  exists a positive constant $C_b$ such that
\begin{equation}
  \label{eq:constanteCb}
\frac{\pi(x)}{\pi(x-u n(x))}  \leq C_b \leq \inf_{z \in \calB_m(x,b)} \frac{q(\expvec{z}{m}, x)}{q(x, \expvec{z}{m})} \eqsp.
\end{equation}
By (\ref{eq:local:q}), Lemma~\ref{lemme:bgaus}
and Lemma~\ref{lem:minoration:probap}, there exist $C, C_b>0$ - independent of
$x \in S_m$ - such that
\begin{equation*}
  \inf_{z \in \calB_m(x,b)} \frac{q(\expvec{z}{m},x)}{q(x,\expvec{z}{m})} \geq
  C^{\dimstate-|m|} \cigaus^{|m|} \cgaus^{-|m|}  \times \inf_{z \in \calB_m(x,b)} \prod_{i \in I_m} \frac{g_{\vigaus}(x_i -
    z_i)}{g_{\vgaus}(x_i - z_i)} \geq C_b\eqsp.
\end{equation*}
By A\ref{hyp:superexp:pi}, we can choose $r$ large enough so that for all
$\|x\| \geq r$, $\pi(x)/\pi(x - u n(x)) \leq C_b$. This yields
(\ref{eq:constanteCb}). Let $z \in W_m(x)$.  Then, $\|z-x\| \leq b$ so that $z
\in \calB_m(x,b)$. Hence, by (\ref{eq:constanteCb}),
$q(\expvec{z}{m},x)/q(x,\expvec{z}{m}) \geq C_b$. In addition,
\begin{equation*}
   \frac{\pi(\expvec{z}{m})}{\pi(x)}  =
  \frac{\pi(\expvec{z}{m})}{\pi(x-u n(x))} \frac{\pi(x-un(x))}{\pi(x)}  \geq \frac{\pi(\expvec{z}{m})}{\pi(x-u n(x))} \frac{1}{C_b} \geq
  \frac{1}{C_b} \eqsp,
\end{equation*}
where in the last inequality we used A\ref{hyp:cone:pi}. Hence,
\[
\frac{\pi(\expvec{z}{m})}{\pi(x)} \frac{q(\expvec{z}{m},x)}{q(x,\expvec{z}{m})}
\geq 1\eqsp,
\]
and $\alpha_\sto(x, \expvec{z}{m})=1$ thus showing the lemma.
\end{proof}

\begin{lemma} \label{lemme:remaining} 
  $ \limsup \limits_{\|x\| \to \infty} \int_{R(x)} q(x,y) \dom(y) < 1$, where
  $\rmd \nu = \sum_m \un_{S_m} \rmd \nu_m$.
\end{lemma}
\begin{proof}
  Let $x \in S_{m_\star}$.  By definition of $\dom$, by
  Lemma~\ref{lemme:bgaus} and by
  Lemma~\ref{lem:minoration:probap}, there exists a constant $C>0$ such that
\begin{align*}
  1- \int_{R(x)} q(x,y) \dom(y)  = \sum_{m \in \M} \int_{A_m(x)} q(x,\expvec{z}{m}) \rmd z
  & \geq \sum_{m \in \M} \cigaus^{|m|}  \mathcal{G}_m(x) \prod_{i \notin I_m} \rho(\mu_i(x))\eqsp,  \\
  & \geq \cigaus^{|m_\star|} \ \mathcal{G}_{m_\star}(x)  \prod_{i \notin I_{m_\star}} \rho(\mu_i(x)) \eqsp, \\
  & \geq C \, \cigaus^{|m_\star|} \ \mathcal{G}_{m_\star}(x) \eqsp,
\end{align*}
where
\[
\mathcal{G}_m(x) \eqdef \int_{A_m(x)} \prod_{i \in I_m} g_{\vigaus}(x_i-y_i)
\rmd y_i\eqsp.
\]
By Lemma~\ref{lemme:cone}, for any $x \in S_{m_\star}$ large enough,
\[
1- \int_{R(x)} q(x,y) \dom(y) \geq C \, \cigaus^{|m_\star|} I_{m_\star}(x)
\]
where, denoting $A-x \eqdef \{z, z+x \in A\}$,
\begin{equation}
\label{eq:dependanceonx}
I_{m_\star}(x)= \int_{W_{m_\star}(x)-x} \left( \prod_{i \in I_{m_\star}}
  g_{\vigaus}(y_{i}) \rmd y_{i} \right) \times \left( \prod_{i \notin I_{m_\star}}
  \delta_0(\rmd y) \right) \eqsp.
\end{equation}
Note that
\begin{equation*}
  W_{m_\star}(x)-x = \{ - u \, n(x) - s \xi; 0<s<b-u,\,\zeta \in S_{m_\star}, \|\zeta\|=1, \|\zeta-n(x)\| \leq \ccone \} \eqsp,
\end{equation*}
so that the integrals in (\ref{eq:dependanceonx}) depend on $x$ only through
$m_\star$. Since $\M$ is finite, there exists a constant $C'>0$
independent of $x$ such that for any $m \in \M$ and $x \in S_m$,
\begin{equation*}
  \int_{W_{m}(x)} \left( \prod_{i \in I_{m}} g_{\vigaus}(x_{i}-y_{i}) \rmd
    y_{i} \right) \times \left( \prod_{i \notin I_{m}} \delta_0(\rmd y) \right) \geq
  C' \eqsp.
\end{equation*}
\end{proof}

\begin{proposition}
  \label{prop:drift}
  $\limsup_{\|x\| \to \infty} PV(x) / V(x) < 1$.
\end{proposition}
\begin{proof}
  The result follows from \eqref{eq:Majoration:Drift} and Lemmas~\ref{lemme:tm}
  and \ref{lemme:remaining}.
\end{proof}

\bibliographystyle{plain}

\end{document}